\documentclass[12pt]{article}

\usepackage{latexsym,amsfonts,amsmath,amsthm,amssymb, epsfig}
\usepackage{color}

\usepackage[vcentering,dvips]{geometry}
\newtheorem{thm}{Theorem}

\newtheorem{cor}{Corollary}

\newtheorem{remark}{Remark}

\theoremstyle{remark}

\newtheorem{exa}{Example}
\theoremstyle{definition}

\newcommand{\R}{\mathbb{R}}

\newcommand{\N}{\mathbb{N}}

\newcommand{\rd}{\,{\rm d}}
\newcommand{\bsa}{{\boldsymbol a}}
\newcommand{\bsx}{{\boldsymbol x}}
\newcommand{\bsy}{{\boldsymbol y}}
\newcommand{\bst}{{\boldsymbol t}}

\newcommand{\uu}{\mathfrak{u}}
\newcommand{\vv}{\mathfrak{v}}
\newcommand{\cP}{\mathcal{P}}

\allowdisplaybreaks

\title{Intractability results for integration in tensor product spaces}  

\author{Erich Novak and  Friedrich Pillichshammer}

\date{}

\begin{document}

\maketitle

\begin{abstract}
In this paper we study lower bounds on the worst-case error of numerical integration 
in tensor product spaces. Thereby integrands are assumed to belong to $d$-fold 
tensor products of spaces of univariate functions. 
As reference we use the $N$-th minimal error of linear rules that use $N$ function values. The information complexity is the minimal number $N$ of function evaluations that is necessary such that the $N$-th minimal error is less than a factor $\varepsilon$ times the initial error, i.e., the error for $N=0$, where $\varepsilon$ belongs to $(0,1)$. We are interested to which extent the information complexity depends on the number $d$ of variables of the integrands.  If the information complexity grows exponentially fast in $d$, then the integration problem is said to suffer from the curse of dimensionality.

Under the assumption of the existence of a worst-case function for the uni-variate problem, which is a function from the considered space whose integral attains the initial error, we present two methods for providing good lower bounds on the information complexity. The first method is based on a suitable decomposition of the worst-case function. This method can be seen as a generalization of the method of decomposable reproducing kernels, that is often successfully applied when integration in Hilbert spaces with a reproducing kernel is studied. 
The second method, although only applicable for positive quadrature rules, 
has the advantage, that it does not require a suitable 
decomposition of the worst-case function. 
Rather, it is based on a spline approximation of the worst-case function
and can be used for analytic functions.  

The methods presented can be applied to problems beyond the Hilbert space setting. For demonstration purposes we apply them to several examples, notably to uniform integration over the unit-cube, weighted integration over the whole space, and integration of infinitely smooth functions over the cube. Some of these results have interesting consequences in discrepancy theory. 
\end{abstract}

\centerline{\begin{minipage}[hc]{130mm}{
}
\end{minipage}}

\section{Introduction}

In this paper we study numerical integration of $d$-variate functions from spaces which are $d$-fold tensor products of identical normed function spaces of uni-variate functions. This is a typical setting that belongs to general linear tensor product problems, that are studied in a multitude of papers (see, e.g., \cite{NW08,NW10} and the references therein).

Let $(F_1,\|\cdot\|_1)$ be a normed space of univariate integrable functions over the Borel measurable domain $D_1 \subseteq \R$ and let, for $d \in \N$, $(F_d,\|\cdot\|_d)$ be a $d$-fold tensor product space $$F_d=\underbrace{F_1\otimes \ldots 
\otimes F_1}_{d-\mbox{{\scriptsize fold}}},$$ which consists of functions over 
the domain $D_d=D_1\times \ldots \times D_1$ ($d$-fold), 
equipped with a crossnorm $\|\cdot\|_d$, i.e., for elementary 
tensors $f(x_1,\ldots,x_d)=f_1(x_1)\cdots f_d(x_d)$ with $f_j:D_1 \rightarrow \R$ 
for $j \in \{1,\ldots,d\}$, we have $\|f\|_d=\|f_1\|_1 \cdots \|f_d\|_1$.
We do not not claim  that the space $F_d$ is uniquely defined by the requirements above.

Consider multivariate integration $$I_d(f):=\int_{D_d} f(\bsx) \rd \bsx \quad \mbox{for $f \in F_d$}.$$ The so-called initial error of this integration problem is 
\begin{equation}\label{def:int:err}
e(0,d)=\|I_d\|=\sup_{f \in F_d\atop \|f\|_d \le 1} \left|I_d(f)\right|.
\end{equation}

We approximate the integrals $I_d(f)$ by algorithms based on $N$ function evaluations of the form 
\begin{equation*}
A_{d,N}(f)=\varphi(f(\bsx_1),f(\bsx_2),\ldots,f(\bsx_N)),
\end{equation*}
where $\varphi:\R^N \rightarrow \R$ is an arbitrary function and where $\bsx_1,\bsx_2,\ldots,\bsx_N$ are points in $D_d$, called integration nodes. Typical examples are linear algorithms of the form 
\begin{equation}\label{def:linAlg}
A_{d,N}(f)=\sum_{k=1}^N a_k f(\bsx_k),
\end{equation}
where $\bsx_1,\bsx_2,\ldots,\bsx_N$ are in $D_d$ and $a_1,a_2,\ldots,a_N$ are real weights that we call integration weights. If $a_1=a_2=\ldots =a_N=1/N$, then the linear algorithm \eqref{def:linAlg} is a so-called quasi-Monte Carlo algorithm, which is denoted by $A_{d,N}^{{\rm QMC}}$.

Define the worst-case error of an algorithm by 
\begin{equation}\label{eq:wce}
e(F_d,A_{d,N})=\sup_{f \in F_d\atop \|f\|_d\le 1} \left|I_d(f)-A_{d,N}(f)\right|.
\end{equation}
It is well known (see \cite{bak} or \cite[Theorem~4.7]{NW08}) that in order to minimize the worst-case error it suffices to consider linear algorithms of the form \eqref{def:linAlg}. So it suffices to restrict the following considerations to linear algorithms. 
We define the $N$-th minimal worst-case error as $$e(N,d):=\min_{A_{d,N}} e(F_d,A_{d,N})$$ where the minimum is extended over all linear algorithms of the form \eqref{def:linAlg} based on $N$ function evaluations along points $\bsx_1,\bsx_2,\ldots,\bsx_N$ from $D_d$ and with real integration weights $a_1,\ldots,a_N$.

For $\varepsilon \in (0,1)$ and $d \in \mathbb{N}$ the information complexity is defined as $$N(\varepsilon,d)=\min\{N \in \mathbb{N} \ : \ e(N,d) \le \varepsilon \, e(0,d)\}.$$ The information complexity of a problem (in the present paper we only consider integration problems) is a well studied concept in information based complexity. It is the minimal number of function evaluations (of ``information'') that is necessary in order to be able to reduce the initial error by a factor of $\varepsilon$, the so-called error demand.

Usually one is interested in the grow rate of $N(\varepsilon,d)$ when the error demand $\varepsilon$ tends to 0 and the dimension $d$ tends to infinity. A problem is called tractable, if $N(\varepsilon,d)$ grows at most sub-exponential in $d$ and in $\varepsilon^{-1}$ when $d \to \infty$ and $\varepsilon \to 0$. In literature many notions of tractability are studied in order to classify the sub-exponential grow rates like, e.g., weak tractability, polynomial tractability, strong polynomial tractabilty. See, for example, the trilogy \cite{NW08,NW10,NW12}. If, however, $N(\varepsilon,d)$ grows at least exponentially fast in $d$ and/or $\varepsilon^{-1}$ for $d \to \infty$ and $\varepsilon \to 0$, then the problem is said to be intractable. If, in particular, $N(\varepsilon,d)$ grows at least exponentially fast in $d$ for a fixed $\varepsilon>0$, then it is said that the problem suffers from the curse of dimensionality. 

For a given (integration) problem it is an important question to decide whether it is tractable or not. In order to verify the possible curse of dimensionality, one needs good lower bounds on the $N$-th minimal error in dimension $d$. This, however, is often a very difficult task. There are some methods like the method of constructing suitable ``fooling-functions'' (also known as ``bump functions'') 
which are functions in the given space, which are zero along a given set of integration nodes, but which have a large value of the integral (see, e.g., the survey \cite[Section~2.7]{DHP15}).

A strong method for proving lower bounds is the ``method of reproducing kernels'' 
that was invented by Novak and Wo\'{z}niakowski in \cite{NW01} and which is presented in detail also in the book \cite[Chapter~11]{NW10}. However, this method can only be applied for integration problems in reproducing kernel Hilbert spaces whose reproducing kernel $K_1$ satisfies the (strong) property of being decomposable, which means that there exists a number $a \in \R$ such that the sets 
$$D_{(0)} :=\{x \in D_1 \ : \ x \le a\}\quad \mbox{and}\quad D_{(1)} := \{x \in D_1 \ : \ x \ge a\}$$
are nonempty and $$K_1(x,t)=0 \quad \mbox{for $(x,t) \in D_{(0)}\times D_{(1)} \cup D_{(1)}\times D_{(0)}$,}$$ or whose reproducing kernel $K_1$ has an additive part, which is decomposable in the above sense (see \cite[Section~11.5]{NW10}).

It is the aim of the present paper to reformulate the method of decomposable kernels in a way such that it can be applied to other settings beyond reproducing kernel Hilbert spaces. Our approach is based on a decomposition of a so-called worst-case function, which is a function in the unit ball of the function space under consideration, whose integral attains the initial error. In general, such a worst-case function does not have to exist. Our results rely on the assumption of the existence of a worst-case function. We will assume that the worst-case function can be decomposed in a suitable way, which mimics the ``kernel-decomposition'' in the method of decomposable kernels. If this is the case, then we can show that the integration problem suffers from the curse of dimensionality. From this point of view our approach might be called the ``method of decomposable worst-case functions''. The method will be presented in Section~\ref{sec:meth1}. We discuss the cases of worst-case functions that are fully decomposable (Section~\ref{sec:meth1.1}) as well as worst-case function which only have a decomposable additive part (Section~\ref{sec:meth2}), where the latter can be seen as generalization of the method for non-decomposable kernels from \cite[Section~11.5]{NW10} beyond the Hilbert space setting.

In Section~\ref{sec:exa} we discuss several applications of the presented method. First we revisit the method of decomposable kernels and show that the new approach indeed comprises the ``method of decomposable kernels'' as a special case. Then we consider uniform integration of functions of smoothness $r$. The special case $r=1$ shows that the $L_p$-discrepancy anchored in a point $a \in (0,1)$ suffers from the curse of dimensionality for all $p \in [1,\infty)$. The same holds true for the $L_p$-quadrant discrepancy at $a \in (0,1)$. Finally, we consider weighted integration over $\R^d$ of functions of smoothness $r$. 

The assumption of having a {\it suitably decomposable} worst-case function is rather strong. Furthermore, one of the conditions for the proposed decomposition is rather hard to check. This problem can be overcome when one restricts the consideration to positive linear rules. This will be done in Section~\ref{sec:meth3}. We propose a spline approximation of the worst-case function which has the 
advantage that it works completely without any decomposition;
hence it can be used for analytic functions.  
Applications of this method are presented in Section~\ref{sec:exa2}, 
which concludes the paper.

Throughout the paper we use the notation $(x)_+=x$ if $x >0$ and 0 if $x \le 0$. 
Furthermore, for $d \in \N$ we use the notation $[d]:=\{1,2,\ldots,d\}$.  

\section{General quadrature rules}\label{sec:meth1}

In the following we use a notation that is very close to the notation in \cite[Section~11.4]{NW10} with the intention to make it easier for the reader to compare our approach to the ``method of decomposable kernels'' as presented in \cite[Section~11.4]{NW10}.

Assume that there exists a worst-case function $h_1 \in F_1$, i.e., a function $h_1 \in F_1$ which satisfies $$e(0,1)=I_1\left(\frac{h_1}{\|h_1\|_1} \right), $$ 
where $e(0,1)=\|I_1\|$ is the initial error for integration in $F_1$ like in \eqref{def:int:err}.
In addition we always assume that the tensor product
$$
h_d(\bsx) :=  \prod_{j=1}^d h_1(x_j)
$$
is a worst-case function for $d>1$ and hence 
$e(0,d) = e(0,1)^d$ is the initial error for integration in $F_d$.

\begin{remark}\rm
As already mentioned, in general a worst-case function does not have to exist. For example take the anchored space $$F_1=\{f:[0,1]\rightarrow \R\ : \ f(0)=0,\, f \, \mbox{ abs. cont. and }\, f' \in L_1([0,1])\},$$ where here and later on ``abs. cont.'' stands for ``absolutely continuous''. As norm on $F_1$ we use $$\|f\|_{1,1}= \int_0^1 |f'(t)| \rd t.$$ Then the norm of the integral operator $I_1$ over $F_1$ is 1, but, because of the anchor-condition and the continuity of the functions in $F_1$, there does not exist a function in the unit ball of $F_1$ whose integral is 1. 
\end{remark}

Examples for integration problems with an existing worst-case function will be given throughout the paper. 

\subsection{Decomposable worst-case functions}\label{sec:meth1.1}

Assume that a worst-case function $h_1$ exists and can be decomposed in the form 
\begin{equation}\label{dech1}
h_1(x)=h_{1,(0)}(x)+h_{1,(1)}(x)\quad \mbox{for $x \in D_1$}
\end{equation}
with the following properties:
\begin{enumerate}
\item[\bf{D1}] $h_{1,(0)}, h_{1,(1)} \in F_1$;
\item[\bf{D2}] there exists a decomposition point $a \in \R$ such that $$D_{(0)}:=\{x \in D_1 \ : \ x \le a\} \quad \mbox{and}\quad D_{(1)}:=\{x \in D_1 \ : \ x \ge a\}$$ are non-empty and the support of $h_{1,(0)}$ is contained in $D_{(0)}$ and the support of $h_{1,(1)}$ is contained in $D_{(1)}$;
\item[\bf{D3}] the integrals $I_1(h_{1,(0)})$ and $I_1(h_{1,(1)})$ are positive; and
\item[\bf{D4}] for every $d \in \N$ and every family $\mathcal{D}$ of subsets of $[d]$ we have 
\begin{eqnarray*}
\left\|\sum_{\uu \in \mathcal{D}}  \prod_{j \in \uu} h_{1,(0)}(x_j) \prod_{j \in [d] \setminus \uu} h_{1,(1)}(x_j)\right\|_d \le  \left\|\sum_{\uu \subseteq [d]} \prod_{j \in \uu} h_{1,(0)}(x_j) \prod_{j \in [d] \setminus \uu} h_{1,(1)}(x_j)\right\|_d.
\end{eqnarray*}
\end{enumerate}

Some remarks are in order:

\begin{remark}\label{re1}\rm
\begin{itemize}
\item Obviously, $D_{(0)} \cup D_{(1)}=D_1$ and  $D_{(0)} \cap D_{(1)} \subseteq \{a\}$.
\item For $d \in \mathbb{N}$ and $\bsx=(x_1,\ldots,x_d) \in D_d$ we have 
\begin{align*}
h_d(\bsx) = & \prod_{j=1}^d h_1(x_j)\\
= & \prod_{j=1}^d(h_{1,(0)}(x_j)+h_{1,(1)}(x_j))\\
= & \sum_{\uu \subseteq [d]} \prod_{j \in \uu} h_{1,(0)}(x_j) \prod_{j \in [d] \setminus \uu} h_{1,(1)}(x_j)
\end{align*}
and, in particular, $$\left\|\sum_{\uu \subseteq [d]} \prod_{j \in \uu} h_{1,(0)}(x_j) \prod_{j \in [d] \setminus \uu} h_{1,(1)}(x_j)\right\|_d = \|h_d\|_d=\|h_1\|_1^d.$$ 
\item In view of the last item property~{\bf D4} can be interpreted as follows: let $g_d$ be any restriction of $h_d$ to arbitrary quadrants $$\bigotimes_{j=1}^d E_j\quad \mbox{where $E_j \in \{D_{(0)},D_{(1)}\}$,}$$ and zero out off these quadrants, then $\|g_d\|_d \le \|h_d\|_d.$ Indeed, later we will use property~{\bf D4} in exactly this vein.
\end{itemize}
\end{remark}

Now we state our first theorem.

\begin{thm}\label{thm1}
Let $(F_1,\|\cdot\|_1)$ be a normed space of univariate functions over $D_1 \subseteq \R$ and let for $d \in \N$, $(F_d,\|\cdot\|_d)$ be the $d$-fold tensor product space equipped with a cross norm $\|\cdot\|_d$. Assume in $F_1$ exists a worst-case function $h_1$ which can be decomposed according to \eqref{dech1} satisfying properties~{\bf D1}-{\bf D4}. Then for the $N$-th minimal integration error in $F_d$ we have 
\begin{equation}\label{er:est:thm1}
e(N,d) \ge (1-N \alpha^d)_+ \, e(0,d),
\end{equation}
 where 
\begin{equation}\label{def:alpha:thm1}
\alpha:= \frac{\max(I_1(h_{1,(0)}),I_1(h_{1,(1)}))}{I_1(h_{1,(0)}) + I_1(h_{1,(1)})} \in \left[\frac{1}{2},1\right).
\end{equation}
Furthermore, for the information complexity we have $$N(\varepsilon,d) \ge \left(\frac{1}{\alpha}\right)^d\, (1-\varepsilon) \quad \mbox{for all $\varepsilon \in (0,1)$ and $d \in \N$.}$$ In particular, the integration problem over $F_d$ suffers from the curse of dimensionality.
\end{thm}

\begin{proof}
Let $N \in \mathbb{N}$ and let $\mathcal{P}=\{\bsx_1,\bsx_2,\ldots,\bsx_N\}$ be a set of $N$ quadrature nodes in $D_d$. 
Now we use {\bf D2}. For this specific $\mathcal{P}$ define the ``fooling function''
\begin{equation}\label{def:fool}
g_d(\bsx)=\sum_{\uu \in \mathcal{D}^*}  \prod_{j \in \uu} h_{1,(0)}(x_j) \prod_{j \in [d] \setminus \uu} h_{1,(1)}(x_j),
\end{equation}
where $\mathcal{D}^\ast=\mathcal{D}^\ast(\cP)$ is the set of all subsets $\uu$ of $[d]$ with the property that for all $k \in \{1,2,\ldots,N\}$ it holds that 
\begin{equation}\label{cond:u}
\bsx_k \not \in  \bigotimes_{j=1}^d E_{j,\uu},
\end{equation}
where 
\begin{equation*}
E_{j,\uu}:= \left\{ 
\begin{array}{ll}
D_{(0)} & \mbox{if $j \in \uu$,}\\
D_{(1)} & \mbox{if $j \in [d]\setminus \uu$.}
\end{array}
\right.
\end{equation*}
Note that $\bigotimes_{j=1}^d E_{j,\uu}$ contains the support of 
\begin{equation}\label{supp:u}
\prod_{j \in \uu} h_{1,(0)}(x_j) \prod_{j \in [d] \setminus \uu} h_{1,(1)}(x_j)
\end{equation}
and hence condition \eqref{cond:u} means that every point from the set $\mathcal{P}$ is outside the support of the function in \eqref{supp:u}. This, however, guarantees that 
$$g_d(\bsx_k)=0 \quad \mbox{for all $k \in \{1,2,\ldots,N\}$.}$$ 
Observe that the fooling function $g_d$ is an element of the $2^d$-dimensional 
(tensor product) subspace generated, for $d=1$, by the two functions 
$h_{1,(0)}$ and $h_{1,(1)}$. Formally, we prove the lower bound for this 
finite-dimensional subspace. 

Since the sum $g_d(\bsx)$ constitutes a partial sum of $h_d(\bsx)$ it follows from property~{\bf D4} that 
\begin{equation}\label{centass}
\|g_d\|_d \le \|h_d\|_d.
\end{equation}
Set $\widetilde{g}_d(\bsx):=\|h_d\|_d^{-1} g_d(\bsx)$ such that $\|\widetilde{g}_d\|_d \le 1$. 

Since $\widetilde{g}_d(\bsx_k)=0$ for all $k \in \{1,2,\ldots,N\}$ we have for any linear algorithm \eqref{def:linAlg} that is based on $\cP$ that
\begin{equation}\label{err_est10}
e(F_d,A_{d,N}) \ge I_d(\widetilde{g}_d) = \frac{1}{\|h_d\|_d} \sum_{\uu \in \mathcal{D}^*} (I_1(h_{1,(0)}))^{|\uu|} (I_1(h_{1,(1)}))^{d-|\uu|}.
\end{equation}
Due to the construction of $g_d$, at most $N$ of the sets $$\bigotimes_{j=1}^d E_{j,\uu} \quad \mbox{with $\uu \subseteq [d]$}$$ can contain a point from the node set $\mathcal{P}$. This implies that 
\begin{eqnarray}\label{eq:sumI1I1}
\lefteqn{\sum_{\uu \in \mathcal{D}^*} (I_1(h_{1,(0)}))^{|\uu|} (I_1(h_{1,(1)}))^{d-|\uu|}}\\
 & \ge & \sum_{\uu \subseteq [d]} (I_1(h_{1,(0)}))^{|\uu|} (I_1(h_{1,(1)}))^{d-|\uu|} - N \max_{\uu \subseteq [d]} (I_1(h_{1,(0)}))^{|\uu|} (I_1(h_{1,(1)}))^{d-|\uu|}\nonumber \\
 & = & \left(I_1(h_{1,(0)}) + I_1(h_{1,(1)})\right)^{d} - N \left( \max(I_1(h_{1,(0)}),I_1(h_{1,(1)}))\right)^{d}\nonumber \\
 & = &  (I_1(h_1))^d (1- N \alpha^d), \nonumber
\end{eqnarray}
with $\alpha$ from \eqref{def:alpha:thm1} (note that $\alpha \in [\tfrac{1}{2},1)$, because of {\bf D3}). Since, obviously, \eqref{eq:sumI1I1} is non-negative, we can re-write the lower bound as 
$$
\sum_{\uu \in \mathcal{D}^*} (I_1(h_{1,(0)}))^{|\uu|} (I_1(h_{1,(1)}))^{d-|\uu|} \ge (I_1(h_1))^d (1- N \alpha^d)_+.
$$
Inserting into \eqref{err_est10} yields
\begin{equation*}
e(N,d) \ge e(F_d,A_{d,N}) \ge \frac{(I_1(h_1))^d}{\|h_d\|_d}  (1-N \alpha^{d})_+ = e(0,d) (1-N \alpha^d)_+,
\end{equation*}
where we used that $(I_1(h_1))^d\|h_d\|_d^{-1}=(I_1(h_1) \|h_1\|_1^{-1})^d=(e(0,1))^d=e(0,d)$. This proves \eqref{er:est:thm1}.

Now let $\varepsilon \in (0,1)$ and assume that $e(N,d) \le \varepsilon \, e(0,d)$. This implies that $(1-N \alpha^d)_+ \le \varepsilon$. If $N < \alpha^{-d}$, then $$(1-N \alpha^d)=(1-N \alpha^d)_+ \le \varepsilon$$ and hence $$N \ge (1- \varepsilon) \left(\frac{1}{\alpha}\right)^d.$$ This lower bound trivially holds true if $N \ge \alpha^{-d}$. For the information complexity this implies $$N(\varepsilon,d) \ge  (1-\varepsilon) \left(\frac{1}{\alpha}\right)^d.$$ Note that $\alpha \in [\tfrac{1}{2},1)$ and hence $1< \frac{1}{\alpha}\le 2$. This finally implies that the integration problem over $F_d$ suffers from the curse of dimensionality.
\end{proof}

Property~{\bf D4} is of course not a very nice condition. 
More desirable would be a condition that has to be stated only in 
the one-dimensional case. It is easy to see that the 
condition $\|h_1\|_1=\|h_{1,(0)}\|_1+\|h_{1,(1)}\|_1$ always 
implies property~{\bf D4}. 
This condition is rather restrictive but can be fulfilled. We give an extreme example. 

\begin{exa}\rm
Let  $F_1$ be a space of dimension 2 of functions over the interval 
$[-1, 1]$ with the $L_1$-norm. 
Take any positive integrable function $h_1$ and define 
$h_{1,(0)}= h_1 \, {\bf 1}_{[-1,0]}$ and $h_{1,(1)}= h_1 \, {\bf 1}_{[0,1]}$. 
Thus properties~{\bf D1}-{\bf D3} are satisfied with $a=0$. 
It turns out that all three functions 
$h_1$, $h_{1,(0)}$ and $h_{1,(1)}$ are worst case functions and we obtain 
$$
e(N,d) =  e(0,d) = 1  \qquad \hbox{\rm for all} \qquad N < 2^d.
$$
This example shows also that the lower bound from the theorem is not optimal. 
\end{exa}

The situation just considered motivates the restriction to norms in $F_d$ with the following property: 

For $q \in [1,\infty]$ we say that $\|\cdot\|_d$ has the ``$q$-property'', if there exists a number $q \in [1,\infty)$ such that for any two functions $f,g \in F_d$ with disjoint support we have 
\begin{equation}\label{qprop}
\|f+g\|_d^q=\|f\|_d^q+\|g\|_d^q,
\end{equation}
or, for $q=\infty$, 
\begin{equation}\label{infty:prop}
\|f+g\|_d=\max(\|f\|_d,\|g\|_d).
\end{equation}
Examples of norms with the $q$-property will be given in Section~\ref{sec:exa}. If the considered norms $\|\cdot\|_d$ satisfy \eqref{qprop} or \eqref{infty:prop}, respectively,  for every $d \in \N$, then property~{\bf D4} follows already from the other assumptions.

\begin{thm}\label{thm2}
Let $(F_1,\|\cdot\|_1)$ be a normed space of univariate functions over $D_1 \subseteq \R$ and let for $d \in \N$, $(F_d,\|\cdot\|_d)$ be the $d$-fold tensor product space equipped with a cross norm $\|\cdot\|_d$. Assume that there exists a $q \in [1,\infty]$ such that $\|\cdot\|_d$ satisfies for all $d\in \N$ the $q$-property. Assume in $F_1$ exists a worst-case function $h_1$ which can be decomposed according to \eqref{dech1} satisfying properties~{\bf D1}-{\bf D3}. Then for the $N$-th minimal integration error in $F_d$ we have $$e(N,d) \ge (1-N \alpha^d)_+ \, e(0,d),$$ where $$\alpha:= \frac{\max(I_1(h_{1,(0)}),I_1(h_{1,(1)}))}{I_1(h_{1,(0)}) + I_1(h_{1,(1)})} \in \left[\frac{1}{2},1\right).$$ Furthermore, for the information complexity we have $$N(\varepsilon,d) \ge (1-\varepsilon) \left(\frac{1}{\alpha}\right)^d \quad \mbox{for all $\varepsilon \in (0,1)$ and $d \in \N$.}$$ In particular, the integration problem over $F_d$ suffers from the curse of dimensionality.
\end{thm}

\begin{proof}
The result follows directly from the proof of Theorem~\ref{thm1}. Define the fooling function $g_d$ like in \eqref{def:fool}. Then it is obvious that the terms $$\prod_{j \in \uu} h_{1,(0)}(x_j) \prod_{j \in [d] \setminus \uu} h_{1,(1)}(x_j) \quad \mbox{for $\uu \subseteq [d]$}$$ have pairwise a disjoint support. From the $q$-property of the norm we obtain   
\begin{align*}
\|g_d\|_d^q = & \sum_{\uu \in \mathcal{D}^*}  \left\| \prod_{j \in \uu} h_{1,(0)}(x_j) \prod_{j \in [d] \setminus \uu} h_{1,(1)}(x_j)\right\|_d^q\\
\le & \sum_{\uu \subseteq [d]}  \left\| \prod_{j \in \uu} h_{1,(0)}(x_j) \prod_{j \in [d] \setminus \uu} h_{1,(1)}(x_j)\right\|_d^q\\
= & \|h_d\|_d^q,
\end{align*}
when $q$ is finite, and similarly for $q=\infty$. Hence, we have $\|g_d\|_d \le \|h_d\|_d$. The remainder of the proof is the same as the proof of Theorem~\ref{thm1}.
\end{proof}

\subsection{Worst-case functions with a decomposable part}\label{sec:meth2}

Decomposability of a worst-case function is a rather strong assumption. However, in some cases one has a more relaxed situation in which a worst-case function has at least an additive part, which is decomposable in the previous sense. In more detail, assume now that a worst-case function $h_1$ exists and can be decomposed in the form 
\begin{equation}\label{dech:part}
h_1(x)=h_{1,1}(x)+h_{1,2,(0)}(x)+h_{1,2,(1)}(x)
\end{equation}
with
\begin{enumerate}
\item[{\bf DP1}] $h_{1,1}, h_{1,2,(0)}, h_{1,2,(1)} \in F_1$;
\item[{\bf DP2}] there exists $a \in \R$ such that $$D_{(0)}=\{x \in D_1 \ : \ x \le a\} \quad \mbox{and}\quad D_{(1)}=\{x \in D_1 \ : \ x \ge a\}$$ are non-empty and the support of $h_{1,2,(0)}$ is contained in $D_{(0)}$ and the support of $h_{1,2,(1)}$ is contained in $D_{(1)}$; 
\item[{\bf DP3}] the integrals $I_1(h_{1,2,(0)})$ and $I_1(h_{1,2,(1)})$ are positive; and
\item[{\bf DP4}] for every $d \in \N$, every $\uu \subseteq [d]$ and every family $\mathcal{D}_{\uu}$ of subsets of $\uu$ we have 
\begin{eqnarray*}
\lefteqn{\left\| \sum_{\uu \subseteq [d]} \prod_{j \not\in \uu} h_{1,1}(x_j) \sum_{\vv \in \mathcal{D}_{\uu}}  \prod_{j \in \vv} h_{1,2,(0)}(x_j) \prod_{j \in \uu \setminus \vv} h_{1,2,(1)}(x_j) \right\|_d}\\
& \le & \left\| \sum_{\uu \subseteq [d]} \prod_{j \not\in \uu} h_{1,1}(x_j) \sum_{\vv \subseteq \uu}\prod_{j \in \vv} h_{1,2,(0)}(x_j) \prod_{j \in \uu \setminus \vv} h_{1,2,(1)}(x_j) \right\|_d.
\end{eqnarray*}
\end{enumerate}

\begin{remark}\rm
For $d \in \mathbb{N}$ and $\bsx=(x_1,\ldots,x_d) \in D_d$ we have 
\begin{align*}
h_d(\bsx)  := & \prod_{j=1}^d h_1(x_j) \\
 = & \prod_{j=1}^d(h_{1,1}(x_j)+h_{1,2,(0)}(x_j)+h_{1,2,(1)}(x_j))\\
 = & \sum_{\uu \subseteq [d]} \prod_{j \not\in \uu} h_{1,1}(x_j) \sum_{\vv \subseteq \uu}\prod_{j \in \vv} h_{1,2,(0)}(x_j) \prod_{j \in \uu \setminus \vv} h_{1,2,(1)}(x_j)
\end{align*}
and, in particular, $$\left\| \sum_{\uu \subseteq [d]} \prod_{j \not\in \uu} h_{1,1}(x_j) \sum_{\vv \subseteq \uu}\prod_{j \in \vv} h_{1,2,(0)}(x_j) \prod_{j \in \uu \setminus \vv} h_{1,2,(1)}(x_j) \right\|_d = \|h_d\|_d=\|h_1\|_1^d.$$ 
\end{remark}

\begin{thm}\label{thm3}
Let $(F_1,\|\cdot\|_1)$ be a normed space of univariate functions over $D_1 \subseteq \R$ and let for $d \in \N$, $(F_d,\|\cdot\|_d)$ be the $d$-fold tensor product space equipped with a crossnorm $\|\cdot\|_d$. Assume in $F_1$ exists a worst-case function $h_1$ which can be decomposed according to \eqref{dech:part} satisfying properties~{\bf DP1}-{\bf DP4}. Then for the $N$-th minimal integration error in $F_d$ we have 
\begin{equation}\label{thm2:a1}
e(N,d) \ge \frac{\alpha_1^d}{\|h_d\|_d} \sum_{k=0}^d {d \choose k} \alpha_3^{k} \left(1-N \alpha^k\right)_+,
\end{equation}
where 
\begin{equation}\label{def2:al}
\alpha:= \frac{\max(I_1(h_{1,2,(0)}),I_1(h_{1,2,(1)}))}{I_1(h_{1,2,(0)}) + I_1(h_{1,2,(1)})} \in \left[\frac{1}{2},1\right),
\end{equation}
and 
\begin{equation}\label{def2:al123}
\alpha_1:=I_1(h_{1,1}),\quad \alpha_2:=I_1(h_{1,2,(0)}) + I_1(h_{1,2,(1)})\quad \mbox{and}\quad \alpha_3:=\frac{\alpha_2}{\alpha_1}>0.
\end{equation}

Furthermore, we have 
\begin{equation}\label{thm2:lim}
\liminf_{d \rightarrow \infty} \frac{e(\lfloor C^d\rfloor,d)}{e(0,d)} =1 \quad \mbox{for all } C \in (1,\alpha^{-\alpha_3/(1+\alpha_3)})
\end{equation}
and for all $\varepsilon \in (0,1)$ there exists a $d_0(\varepsilon)\in \N$ such that  $$N(\varepsilon,d) \ge \lfloor C^d \rfloor \quad \mbox{for all $d \ge d_0(\varepsilon)$}$$ with $C$ like in \eqref{thm2:lim}. In particular, the integration problem over $F_d$ suffers from the curse of dimensionality.
\end{thm}

\begin{proof}
Let $N \in \mathbb{N}$ and let $\mathcal{P}=\{\bsx_1,\bsx_2,\ldots,\bsx_N\}$ be a set of $N$ quadrature nodes in $D_d$. For $\uu \subseteq [d]$ let $\bsx_{k,\uu}$ denote the projection of $\bsx_k$ to the coordinates $j$ in $\uu$. For this specific $\mathcal{P}$ define the ``fooling function''
\begin{equation*}
g_d(\bsx)=\sum_{\uu \subseteq [d]} \prod_{j \not\in \uu} h_{1,1}(x_j) \sum_{\vv \in \mathcal{D}^*_{\uu}}  \prod_{j \in \vv} h_{1,2,(0)}(x_j) \prod_{j \in \uu \setminus \vv} h_{1,2,(1)}(x_j),
\end{equation*}
where, for every $\uu \subseteq [d]$, the set $\mathcal{D}^{\ast}_{\uu}=\mathcal{D}^{\ast}_{\uu}(\cP)$ is the set of all subsets $\vv$ of $\uu$ with the property that for all $k \in \{1,2,\ldots,N\}$ it holds that 
\begin{equation}\label{condsup2}
\bsx_{k,\uu} \not \in  \bigotimes_{j \in \uu} E_{j,\vv},
\end{equation}
where 
$$E_{j,\vv}=\left\{ 
\begin{array}{ll}
D_{(0)} & \mbox{if $j \in \vv$,}\\
D_{(1)} & \mbox{if $j \in \uu \setminus \vv$.}
\end{array}
\right.
$$
Note that $\bigotimes_{j \in \uu} E_{j,\vv}$ contains the support of $\prod_{j \in \vv} h_{1,2,(0)}(x_j) \prod_{j \in \uu \setminus \vv} h_{1,2,(1)}(x_j).$ Hence, condition \eqref{condsup2} means that every point from the set $\mathcal{P}$ is outside the support of $\prod_{j \not\in \uu} h_{1,1}(x_j)\prod_{j \in \vv} h_{1,2,(0)}(x_j) \prod_{j \in \uu \setminus \vv} h_{1,2,(1)}(x_j)$ and this guarantees that 
$$g_d(\bsx_k)=0 \quad \mbox{for all $k \in \{1,2,\ldots,N\}$.}$$ 
Observe that the fooling function $g_d$ is an element of the $3^d$-dimensional 
(tensor product) subspace generated, for $d=1$, by the three functions 
$h_{1,1}$, $h_{1,2,(0)}$ and $h_{1,2,(1)}$. Formally, we prove the lower bound for this 
finite-dimensional subspace.

Since $g_d(\bsx)$ constitutes a partial sum of $h_d(\bsx)$ it follows from property~{\bf DP4} that
\begin{equation*}
\|g_d\|_d \le \|h_d\|_d.
\end{equation*}
Set $\widetilde{g}_d(\bsx):=\|h_d\|_d^{-1} g_d(\bsx)$ such that $\|\widetilde{g}_d\|_d \le 1$. 

Since $\widetilde{g}_d(\bsx_k)=0$ for all $k \in \{1,2,\ldots,N\}$ we have for any linear algorithm \eqref{def:linAlg} that is based on $\cP$ that
\begin{align}\label{err_est1}
e(F_d,A_{d,N}) \ge & I_d(\widetilde{g}_d)\nonumber\\
 = & \frac{1}{\|h_d\|_d} \sum_{\uu \subseteq [d]} (I_1(h_{1,1}))^{d-|\uu|} \nonumber \\
 & \hspace{1.8cm}\times \sum_{\vv \in \mathcal{D}^{*}_{\uu}}  (I_1(h_{1,2,(0)}))^{|\vv|} (I_1(h_{1,2,(1)}))^{|\uu|-|\vv|}.
\end{align}

Now, for fixed $\uu \subseteq [d]$, we estimate the inner-most sum in \eqref{err_est1}. Due to the construction of $g_d$, at most $N$ of the $2^{|\uu|}$ sets $$\bigotimes_{j \in \uu}E_{j,\vv} \quad \mbox{with $\vv \subseteq \uu$}$$ can contain a point from the node set $\mathcal{P}$. This implies that 
\begin{eqnarray}\label{eq:I1I1u}
\lefteqn{\sum_{\vv \in \mathcal{D}^{*}_{\uu}} (I_1(h_{1,2,(0)}))^{|\vv|} (I_1(h_{1,2,(1)}))^{|\uu|-|\vv|}}\\
 & \ge & \sum_{\vv \subseteq \uu} (I_1(h_{1,2,(0)}))^{|\vv|} (I_1(h_{1,2,(1)}))^{|\uu|-|\vv|} - N \max_{\vv \subseteq \uu} (I_1(h_{1,2,(0)}))^{|\vv|} (I_1(h_{1,2,(1)}))^{|\uu|-|\vv|}\nonumber \\
 & = & \left(I_1(h_{1,2,(0)}) + I_1(h_{1,2,(1)})\right)^{|\uu|} - N \left( \max(I_1(h_{1,2,(0)}),I_1(h_{1,2,(1)}))\right)^{|\uu|}\nonumber \\
 & = &  \alpha_2^{|\uu|} (1- N \alpha^{|\uu|}), \nonumber 
\end{eqnarray}
with $\alpha$ from \eqref{def2:al} (note that $\alpha \in [\tfrac{1}{2},1)$ because of {\bf DP3}) and $\alpha_2$ from \eqref{def2:al123}.  Since, obviously, \eqref{eq:I1I1u} is non-negative, we can re-write the lower bound as 
$$
\sum_{\vv \in \mathcal{D}^{*}_{\uu}} (I_1(h_{1,2,(0)}))^{|\vv|} (I_1(h_{1,2,(1)}))^{|\uu|-|\vv|}\ge \alpha_2^{|\uu|} (1- N \alpha^{|\uu|})_+. 
$$
Inserting into \eqref{err_est1}, thereby using the notation $\alpha_1$ and $\alpha_3$ from \eqref{def2:al123} we obtain
\begin{equation*}
e(N,d) \ge e(F_d,A_{d,N}) \ge \frac{\alpha_1^d}{\|h_d\|_d}  \sum_{k=0}^d {d \choose k} \alpha_3^{k} (1- N \alpha^k)_+
\end{equation*}
and this finishes the proof of \eqref{thm2:a1}.

For the normalized $N$-th minimal error we therefore obtain
\begin{equation}\label{lb:ne}
\frac{e(N,d)}{e(0,d)} \ge \left(\frac{\alpha_1}{I_1(h_1)}\right)^d  \sum_{k=0}^d {d \choose k} \alpha_3^{k} (1-N \alpha^k)_+,
\end{equation}
where we used that $\|h_d\|_d \, e(0,d)=\left(\|h_1\|_1 \, e(0,1)\right)^d =(I_1(h_1))^d$.

Now we proceed using ideas from \cite[p.~185]{NW10}. Let $b \in (0,1)$ and let $N =\lfloor C^d\rfloor$ with $C \in (1,\alpha^{-\alpha_3/(1+\alpha_3)})$. This means that $C \alpha^{\alpha_3/(1+\alpha_3)} < 1$. Then there exists a positive $c$ such that $$c < \frac{\alpha_3}{1+\alpha_3} \quad \mbox{ and }\quad C \alpha^c<1.$$ Put $k(d):= \lfloor c d\rfloor$. For sufficiently large $d$ we have 
\begin{equation}\label{eq:b}
N \alpha^k \le C^d \alpha^{cd-1}=\alpha^{-1} (C \alpha^c)^d \le b \quad \mbox{ for all $k \in (k(d),d]$.}
\end{equation}
Put further $C_{d,k}:={d \choose k} \alpha_3^k$. Then \eqref{lb:ne} and \eqref{eq:b} imply that
\begin{align*}
\frac{e(\lfloor C^d\rfloor,d)}{e(0,d)} \ge & \left(\frac{\alpha_1}{I_1(h_1)}\right)^d \sum_{k=k(d)+1 }^d C_{d,k} (1-N \alpha^k)_+\\
\ge &  (1-b)\left(\frac{\alpha_1}{I_1(h_1)}\right)^d \sum_{k=k(d)+1 }^d C_{d,k} \\
= &  (1-b) \left(\frac{\alpha_1}{I_1(h_1)}\right)^d \left(\sum_{k=0}^d C_{d,k} - \sum_{k=0}^{k(d)} C_{d,k} \right).
\end{align*}
Note that $$\sum_{k=0}^d C_{d,k} = (1+\alpha_3)^d =\frac{(\alpha_1+\alpha_2)^d}{\alpha_1^d}=\left(\frac{I_1(h_1)}{\alpha_1}\right)^d.$$ 
This and $e(0,1)=I_1(h_1)/\|h_1\|_1$ yields
\begin{align*}
\frac{e(\lfloor C^d\rfloor,d)}{e(0,d)} \ge &  (1-b)  \left(1 -\frac{\sum_{k=0}^{k(d)} C_{d,k}}{(1+\alpha_3)^d} \right).
\end{align*} 

Now put $$\alpha(d):=\frac{\sum_{k=0}^{k(d)} C_{d,k}}{(1+\alpha_3)^d}.$$  It is shown in \cite[p.~185]{NW10} that $\alpha(d)$ tends to zero when $d \rightarrow \infty$.  Hence, for any positive $\delta$ we can find an integer $d(\delta)$ such that for all $d \ge d(\delta)$ we have $$1 \ge  \frac{e(\lfloor C^d\rfloor,d)}{e(0,d)} \ge (1-b)(1-\delta).$$ Since $b$ and $\delta$ can be arbitrarily close to zero, it follows that $$\liminf_{d \rightarrow \infty} \frac{e(\lfloor C^d\rfloor,d)}{e(0,d)} = 1, $$ and this holds true for all $C \in (1,\alpha^{-\alpha_3/(1+\alpha_3)})$. This finishes the proof of \eqref{thm2:lim}.

Finally, take $\varepsilon \in (0,1)$. Then there exists a $d_0(\varepsilon) \in \N$ such that for all $d \ge d_0(\varepsilon)$ it holds that $$\varepsilon \le \frac{e(\lfloor C^d\rfloor,d)}{e(0,d)}$$ and therefore $$N(\varepsilon,d) \ge \lfloor C^d \rfloor \quad \mbox{for all $d\ge d_0(\varepsilon)$}.$$ Since $C>1$ this means the curse of dimensionality. 
\end{proof}

\section{Applications: Part 1}\label{sec:exa}

An important application of the method of decomposable worst-case functions can be found in \cite{NP23} where we proved that the classical $L_p$-star discrepancy suffers from the curse of dimensionality for a sequence of parameters $p$ in $(1,2]$. By nature, there the most difficult part is a verification of property~{\bf DP4}. In this section we present further concrete applications of the method. In the first application we show that the presented method is essentially a generalization of the ``method of decomposable kernels''.

\subsection{The method of decomposable kernels revisited} Let $q=2$ and assume that $F_1$ is a reproducing kernel Hilbert space of absolutely integrable functions over $D_1$ with inner product $\langle \cdot , \cdot \rangle_1$ and norm $\|\cdot\|_1 = \langle \cdot , \cdot \rangle_1^{1/2}$. Assume that the kernel $K_1$ is decomposable, i.e., there exists an element $a \in \R$ such that the sets $$D_{(0)} :=\{x \in D_1 \ : \ x \le a\}\quad \mbox{and}\quad D_{(1)} := \{x \in D_1 \ : \ x \ge a\}$$
are nonempty and $$K_1(x,t)=0 \quad \mbox{for $(x,t) \in D_{(0)}\times D_{(1)} \cup D_{(1)}\times D_{(0)}$.}$$ Compare with \cite[Section~11.4]{NW10}. 

The $d$-fold tensor product space $F_d$ is again a reproducing kernel Hilbert space with kernel $K_d$, inner product $\langle \cdot , \cdot \rangle_d$ and norm $\|\cdot \|_d= \langle \cdot , \cdot \rangle_d^{1/2}$. It is well known that $K_d(\bsx,\bsy)= \prod_{j=1}^d K_1(x_j,y_j)$ for $\bsx=(x_1,\ldots,x_d)$ and $\bsy=(y_1,,\ldots,y_d)$ in $D_d:=D_1^d$. Furthermore, for elementary tensors $f(\bsx)=f_1(x_1)\cdots f_d(x_d)$ and $g(\bsx)=g_1(x_1)\cdots g_d(x_d)$ we have 
\begin{equation}\label{innprod:eltens}
\langle f,g \rangle_d=\prod_{j=1}^d \langle f_j,g_j \rangle_1.
\end{equation}
See \cite{Aron50} for an introduction into the theory of reproducing kernel Hilbert spaces.

It is well known that the worst-case function of the integration problem in $F_1$ is given by $$h_1(x)=\frac{\int_{D_1} K_1(x,t) \rd t}{\left\|\int_{D_1} K_1(\cdot,t) \rd t\right\|_1}.$$
Now we use the decomposition $$h_{1,(0)}(x)=\frac{\int_{D_{(0)}} K_1(x,t) \rd t}{\left\|\int_{D_1} K_1(\cdot,t) \rd t\right\|_1}\quad\mbox{and}\quad h_{1,(1)}(x)=\frac{\int_{D_{(1)}} K_1(x,t) \rd t}{\left\|\int_{D_1} K_1(\cdot,t) \rd t\right\|_1}.$$ Then, obviously $h_1(x)=h_{1,(0)}(x)+h_{1,(1)}(x)$ where $h_{1,(0)},h_{1,(1)} \in F_1$ and ${\rm supp}\, h_{1,(0)} \subseteq D_{(0)}$ and ${\rm supp}\, h_{1,(1)} \subseteq D_{(1)}$. We assume that $I_1(h_{1,(0)})>0$ and $I_1(h_{1,(1)})>0$ such that our decomposition satisfies properties~{\bf D1}-{\bf D3}. 

Finally, $h_{1,(0)}$ and $h_{1,(1)}$ are orthogonal, because 
\begin{align*}
\left\langle \int_{D_{(0)}} K_1(\cdot,t)\rd t, \int_{D_{(1)}} K_1(\cdot,s)\rd s\right\rangle_1  = & \int_{D_{(0)}} \int_{D_{(1)}}  \langle K_1(\cdot,t),K_1(\cdot,s)\rangle_1 \rd t\rd s\\
= & \int_{D_{(0)}} \int_{D_{(1)}}  K_1(s,t) \rd t\rd s = 0.
\end{align*}
Then for every family $\mathcal{D}$ of subsets of $[d]$ we have
\begin{align*}
\lefteqn{\left\|\sum_{\uu \in \mathcal{D}} \prod_{j \in \uu} h_{1,(0)}(x_j) \prod_{j \in [d] \setminus \uu} h_{1,(1)}(x_j)\right\|_{d}^2}\\
&=   \left\langle \sum_{\uu \in \mathcal{D}} \prod_{j \in \uu} h_{1,(0)}(x_j) \prod_{j \in [d] \setminus \uu} h_{1,(1)}(x_j), \sum_{\vv \in \mathcal{D}} \prod_{j \in \vv} h_{1,(0)}(x_j) \prod_{j \in [d] \setminus \vv} h_{1,(1)}(x_j) \right\rangle_d\\
& =  \sum_{\uu \in \mathcal{D}}\sum_{\vv \in \mathcal{D}} \left\langle \prod_{j \in \uu} h_{1,(0)}(x_j) \prod_{j \in [d] \setminus \uu} h_{1,(1)}(x_j),  \prod_{j \in \vv} h_{1,(0)}(x_j) \prod_{j \in [d] \setminus \vv} h_{1,(1)}(x_j) \right\rangle_d\\
&= \sum_{\uu \in \mathcal{D}} \left\langle \prod_{j \in \uu} h_{1,(0)}(x_j) \prod_{j \in [d] \setminus \uu} h_{1,(1)}(x_j),  \prod_{j \in \uu} h_{1,(0)}(x_j) \prod_{j \in [d] \setminus \uu} h_{1,(1)}(x_j) \right\rangle_d\\
& =   \sum_{\uu \in \mathcal{D}} \left\|\prod_{j \in \uu} h_{1,(0)}(x_j) \prod_{j \in [d] \setminus \uu} h_{1,(1)}(x_j)\right\|_d^2\\
& \le \sum_{\uu \subseteq [d]}  \left\|\prod_{j \in \uu} h_{1,(0)}(x_j) \prod_{j \in [d] \setminus \uu} h_{1,(1)}(x_j)\right\|_d^2
\end{align*}
where we used the orthogonality of $h_{1,(0)}$ and $h_{1,(1)}$ and the property \eqref{innprod:eltens} of the inner product for elementary tensors. This shows that also property~{\bf D4} is satisfied.
So we meet all properties~{\bf D1}-{\bf D4} and hence, according to Theorem~\ref{thm1} we have  
$$e(N,d) \ge (1-N \alpha^d)_+ \, e(0,d),$$ where $$\alpha:= \frac{\max(I_1(h_{1,(0)}),I_1(h_{1,(1)}))}{I_1(h_{1,(0)}) + I_1(h_{1,(1)})} \in \left[\frac{1}{2},1\right),$$ and $$N(\varepsilon,d) \ge (1-\varepsilon) \left(\frac{1}{\alpha}\right)^d \quad \mbox{for all $\varepsilon \in (0,1)$ and $d \in \N$,}$$ such that the integration problem over $F_d$ suffers from the curse of dimensionality.

\subsection{Uniform integration of functions of smoothness $r$}\label{susec:3.2}
For $q \in (1,\infty]$,  $r \in \N$, and $a \in (0,1)$ let 
\begin{align*}
F_1=W_{a,q}^{r}([0,1])=\{f:[0,1] \rightarrow \R \ : \ & f^{(j)}(a)=0\, \mbox{ for $0 \le j <r$},\\
& f^{(r-1)}\, \mbox{ abs. cont. and }\, f^{(r)} \in L_q([0,1])\}.
\end{align*}
As norm on $F_1$ we use $$\|f\|_{1,q}:=\|f^{(r)}\|_{L_q}.$$ 
We study integration in the $d$-fold tensor product space $F_d=W_{a,q}^{(r,r,\ldots,r)}([0,1]^d)$, equipped with the crossnorm $\|f\|_{d,q}:=\|f^{(r,r,\ldots,r)}\|_{L_q}$. For $q=2$ and $r=1$ it is shown in \cite[Sec.~11.4.2, pp.~173ff]{NW10} that the integration problem suffers from the curse of dimensionality. With our method we are now able to tackle the general case $q \in (1,\infty]$ and $r \in \N$.

First consider $q \in (1,\infty)$ and observe that the norm $\|\cdot\|_{1,q}$ as well as the crossnorm $\|\cdot\|_{d,q}$ satisfy the $q$-property \eqref{qprop} with the same parameter $q$ like in the definition of the norm. That is, we aim to apply Theorem~\ref{thm2}.

Next we determine the worst-case function. For $f \in F_1$ and for $x \in [0,1]$ we have the representation
\begin{equation*}
f(x)=\int_a^x f^{(r)}(t) \frac{(x-t)^{r-1}}{(r-1)!} \rd t.
\end{equation*}
Using this representation it is elementary to prove that 
$$\int_0^1 f(x) \rd x =  \int_0^1 f^{(r)}(t) g(t) \rd t,$$ where 
$$g(t)=\left\{ 
\begin{array}{rl}
(-t)^r/r! & \mbox{if $t \in [0,a)$,}\\
(1-t)^r/r! &\mbox{if $t \in [a,1]$.}
\end{array}\right.
$$
Applying H\"older's inequality we obtain 
$$\left|\int_0^1 f(x) \rd x \right| \le \|f^{(r)}\|_{L_q} \|g\|_{L_p}=\|f\|_{1,q} \|g\|_{L_p},$$ where $p$ is the H\"older conjugate of $q$, i.e., $1/p+1/q=1$, with equality if $|f^{(r)}(t)|=c |g(t)|^{p-1}$ for some $c >0$. This holds for 
\begin{equation}\label{wc:fct:unifint}
f(t) :=\frac{(-1)^{r+1}}{(r!)^{p-1}}\times \left\{ 
\begin{array}{ll}
\left(\frac{a^{r p}-t^{r p}}{\prod_{i=0}^{r-1}(r p-i)} +\sum_{j=1}^{r-1} \frac{(t-a)^j a^{rp-j}}{j! \prod_{i=j}^{r-1}(r p-i)}\right) & \mbox{if $t \in [0,a)$,}\\[1em]
\left(\frac{(1-a)^{r p}-(1-t)^{r p}}{\prod_{i=0}^{r-1}(r p-i)} + \sum_{j=1}^{r-1} \frac{(t-a)^j (1-a)^{rp-j}}{j! \prod_{i=j}^{r-1}(r p-i)}(-1)^{j}\right) &\mbox{if $t \in [a,1]$.}
\end{array}\right.
\end{equation} 
This means that $h_1:=f$ is the worst-case function (note that also $h_1^{(j)}(a)=0$ for $0 \le j < r$).

We use the decomposition of $h_1$ with $$h_{1,(0)}(x)={\bf 1}_{[0,a]}(x)h_1(x) \quad \mbox{and}\quad h_{1,(1)}(x)={\bf 1}_{[a,1]}(x) h_1(x).$$ In particular, we have $D_1=[0,1]$, $D_{(0)}=[0,a]$ and $D_{(1)}=[a,1]$. Then 
\begin{align*}
I_1( h_{1,(0)})=\frac{(-1)^{r+1} a^{rp+1}}{(r!)^{p-1}\prod_{i=0}^{r-1}(r p-i)}  \left(\frac{rp}{r p+1}+\sum_{j=1}^{r-1} \frac{(-1)^j}{(j+1)!} \prod_{i=0}^{j-1}(r p-i)\right)
\end{align*}
and
\begin{align*}
I_1( h_{1,(1)})= \frac{(-1)^{r+1}(1-a)^{rp+1}}{(r!)^{p-1} \prod_{i=0}^{r-1}(rp-i)}\left(\frac{rp}{rp+1} + \sum_{j=1}^{r-1} \frac{(-1)^j}{(j+1)!} \prod_{i=0}^{j-1}(r p-i)\right).
\end{align*}
It can be shown that 
$$\frac{rp}{r p+1}+\sum_{j=1}^{r-1} \frac{(-1)^j}{(j+1)!} \prod_{i=0}^{j-1}(r p-i) \ \left\{ 
\begin{array}{ll}
>0 & \mbox{if $r$ is odd,}\\
<0 & \mbox{if $r$ is even.} 
\end{array}
\right.
$$
Hence, in any case, we have $$I_1(h_{1,(0)})>0 \quad \mbox{and}\quad I_1(h_{1,(1)})>0.$$

So we meet all required properties {\bf D1}-{\bf D3} and  $$\frac{1}{\alpha}=\frac{I_1(h_{1,(0)}) + I_1(h_{1,(1)})}{\max(I_1(h_{1,(0)}),I_1(h_{1,(1)}))}=1+\left(\frac{1}{\max(a,1-a)}-1\right)^{rp+1}>1$$ and Theorem~\ref{thm2} can be applied.

For the case $q=\infty$ we can argue similarly. Note that the norm $\|\cdot\|_{1,\infty}$ as well as the crossnorm $\|\cdot\|_{d,\infty}$ satisfy the $\infty$-property \eqref{infty:prop} and hence Theorem~\ref{thm2} still can be applied. The H\"older conjugate of $q=\infty$ is $p=1$. A brief consideration shows that the worst-case function for $q=\infty$ is still $h_1=f$ with $f$ from\eqref{wc:fct:unifint} with $p=1$. The remainder works in the same way as for finite $q$.

Now Theorem~\ref{thm2} implies the following result:

\begin{cor}\label{cor1}
For $q\in (1,\infty]$, $r \in \N$ and $a\in (0,1)$ integration in the tensor product space $W_{a,q}^{(r,r,\ldots,r)}([0,1]^d)$ suffers from the curse of dimensionality. In more detail, we have $$N(\varepsilon,d)\ge C_{p,r}^d\, (1-\varepsilon)\quad \mbox{for all $\varepsilon \in (0,1)$ and $d \in \N$,}$$ where $p$ is the H\"older conjugate of $q$ and where 
\begin{equation}\label{def:Cpr}
C_{p,r}=1+\left(\frac{1}{\max(a,1-a)}-1\right)^{r p+1} >1.
\end{equation}
\end{cor}

That is, we generalized the example in \cite[Sec.~11.4.2, pp.~173ff]{NW10} from $q=2$ and $r=1$ to general $q \in (1,\infty]$ and $r \in \N$.

\begin{exa}\label{ex:r1}\rm
In view of the next two applications and also in order to facilitate the comparison with the result in \cite[Sec.~11.4.2, pp.~173ff]{NW10} we briefly specify the case of smoothness $r=1$ explicitly. Here we have
$$h_1(x):=\boldsymbol{1}_{[0,a]}(x) \frac{a^p-x^p}{p} + \boldsymbol{1}_{[a,1]}(x) \frac{(1-a)^{p}-(1-x)^{p}}{p},$$ 
with $$I_1(h_1)  = \frac{a^{p+1}+(1-a)^{p+1}}{p+1} \quad \mbox{and} \quad \|h_1\|_{1,q} = \left(\frac{a^{p+1}+(1-a)^{p+1}}{p+1}\right)^{1/q},$$
such that $$e(0,1)=\left(\frac{a^{p+1}+(1-a)^{p+1}}{p+1}\right)^{1/p}.$$
A suitable decomposition of $h_1$ is given 
by  $$h_{1,(0)}(x)={\bf 1}_{[0,a]}(x) \frac{a^{p}-x^{p}}{p} \quad \mbox{and} \quad h_{1,(1)}(x)={\bf 1}_{[a,1]}(x) \frac{(1-a)^{p}-(1-x)^{p}}{p}.$$ Then $$I_1( h_{1,(0)})=\frac{a^{p+1}}{p+1}>0 \quad \mbox{and} \quad I_1( h_{1,(1)})=\frac{(1-a)^{p+1}}{p+1}>0$$ and $$\frac{1}{\alpha}=\frac{I_1(h_{1,(0)}) + I_1(h_{1,(1)})}{\max(I_1(h_{1,(0)}),I_1(h_{1,(1)}))}=1+\left(\frac{1}{\max(a,1-a)}-1\right)^{p+1} >1.$$ For example, for $a=1/2$ we have $1/\alpha=2$.
\end{exa}

\subsection{$L_p$-discrepancy anchored at $a$}\label{subsec3.3}

For a point set $\{\bsy_1,\ldots,\bsy_N\}$ in $[0,1)^d$, for $a\in [0,1]$ and $p \in [1,\infty)$, the $L_p$-discrepancy anchored at $a$ is defined as $$L_p^{\pitchfork,(a)}(\{\bsy_1,\ldots,\bsy_N\}):=\left(\int_{[0,1]^d} \left|\frac{1}{N}\sum_{k=1}^N {\bf 1}_{J(\bst)}(\bsy_k) - {\rm Vol}(J(\bst))\right|^p \rd \bst \right)^{1/p},$$ where, for $\bst=(t_1,\ldots,t_d)$ in $[0,1]^d$, the test set $J(\bst)$ is defined as $$J(\bst) =J(t_1)\times \ldots \times J(t_d)$$ with $J(t)=[\min(t,a),\max(t,a))$ for $t \in [0,1]$ and ${\rm Vol}(\cdot)$ denotes the volume in $\R^d$. See \cite[Section~9.5.3]{NW10} and \cite[Example~1]{HW12}.

Consider integration in the tensor product space $W_{a,q}^{(r,r,\ldots,r)}([0,1]^d)$ from the previous section. For $r=1$ it can be shown that the worst-case error of a quasi-Monte Carlo rule 
\begin{equation}\label{QMC}
A_{d,N}^{{\rm QMC}}(f)=\frac{1}{N}\sum_{k=1}^N f(\bsx_k) \quad \mbox{for $f \in W_{a,q}^{(1,1,\ldots,1)}([0,1]^d)$,}
\end{equation}
which is based on a point set $\{\bsx_1,\ldots,\bsx_N\}$, is exactly the $L_p$-discrepancy anchored at $a$ of the point set $\{\bsy_1,\ldots,\bsy_N\}$, where $\bsy_k:=\bsa-\bsx_k \pmod{1}$ and where $\bsa=(a,\ldots,a)$, i.e., $$e(W_{a,q}^{(1,1,\ldots,1)}([0,1]^d),A_{d,N}^{{\rm QMC}})=L_p^{\pitchfork,(a)}(\{\bsy_1,\ldots,\bsy_N\}),$$ where $1/p+1/q=1$. See \cite[Section~9.5.3]{NW10} and \cite[Sec.~11.4.3]{NW10}. The initial $L_p$-discrepancy anchored at $a$ is $$L_p^{\pitchfork,(a)}(\emptyset)=\left(\frac{a^{p+1}+(1-a)^{p+1}}{p+1}\right)^{d/p}=e(0,d),$$ where the latter is the initial error of integration in $W_{a,q}^{(1,1,\ldots,1)}([0,1]^d)$.

From Corollary~\ref{cor1} we therefore obtain:

\begin{cor}\label{cor2}
For $a \in (0,1)$ and for $p \in [1,\infty)$ the $L_p$-discrepancy anchored at $a$ suffers from the curse of dimensionality, i.e., 
\begin{align*}
\lefteqn{N_{L_p^{\pitchfork, (a)}}(\varepsilon,d)}\\
&:= \min\{N \in \N\ : \ \exists \, \cP \subseteq [0,1)^d \mbox{ s.t. } |\cP|=N \mbox{ and } L_p^{\pitchfork,(a)}(\cP)\le \varepsilon L_p^{\pitchfork,(a)}(\emptyset)\},
\end{align*}
grows exponentially fast with the dimension $d$. More detailed, $$N_{L_p^{\pitchfork,(a)}}(\varepsilon,d) \ge C_{p,1}^d\, (1-\varepsilon) \quad \mbox{for all $\varepsilon \in (0,1)$ and $d \in \N$,}$$ where $C_{p,1}>1$ is like in \eqref{def:Cpr} with $r=1$.
\end{cor}

We point out that in this result even the endpoint $p=1$ is included. For $a=0$ the $L_p$-discrepancy anchored in 0 is the same as the usual $L_p$-star discrepancy. This case is not covered by  Corollary~\ref{cor2}. Proving the curse of dimensionality for the $L_p$-star discrepancy for finite $p$ beyond $p=2$ was a long standing open question which was recently solved in \cite{NP24}, however, only for $p \in (1,\infty)$. The quest about the curse of dimensionality for the $L_1$-star discrepancy is still open.

\subsection{$L_p$-quadrant discrepancy at $a$}\label{subsec3.4}

For a point set $\{\bsy_1,\ldots,\bsy_N\}$ in $[0,1)^d$, for $a \in [0,1]$ and $p \in [1,\infty)$, the $L_p$-quadrant discrepancy at $a$ is defined as $$L_p^{\boxplus,(a)}(\{\bsy_1,\ldots,\bsy_N\}):=\left(\int_{[0,1]^d} \left|\frac{1}{N}\sum_{k=1}^N {\bf 1}_{Q(\bst)}(\bsy_k) - {\rm Vol}(Q(\bst))\right|^p \rd \bst \right)^{1/p},$$ where, for $\bst=(t_1,\ldots,t_d)$ in $[0,1]^d$, the test set $Q(\bst)$ is defined as $$Q(\bst) =Q(t_1)\times \ldots \times Q(t_d)$$ with $Q(t)=[0,t)$ if $t<a$, and $Q(t)=[t,1)$ if $t \ge a$. For $a=1/2$ and $p=2$ this discrepancy was studied by Hickernell~\cite{hick98} under the name {\it centered discrepancy}. See also \cite[Section~9.5.4]{NW10}, \cite[Example~2]{HW12}, and \cite{HSW2002}.

As before, consider integration in the tensor product space $W_{a,q}^{(1,1,\ldots,1)}([0,1]^d)$. Then the worst-case error of a quasi-Monte Carlo rule \eqref{QMC} for $f \in W_{a,q}^{(1,1,\ldots,1)}([0,1]^d)$ which is based on a point set $\{\bsx_1,\ldots,\bsx_N\}$, is exactly the $L_p$-quadrant discrepancy at $a$ of the same point set, i.e., $$e(W_{a,q}^{(1,1,\ldots,1)}([0,1]^d),A_{d,N}^{{\rm QMC}})=L_p^{\boxplus,(a)}(\{\bsx_1,\ldots,\bsx_N\}),$$ where $1/p+1/q=1$. See \cite[Sec.~9.5.4]{NW10} (for the case $p=2$). The initial $L_p$-quadrant discrepancy at $a$ is $$L_p^{\boxplus,(a)}(\emptyset)=\left(\frac{a^{p+1}+(1-a)^{p+1}}{p+1}\right)^{d/p}=e(0,d),$$ where the latter is the initial error of integration in $W_{a,q}^{(1,1,\ldots,1)}([0,1]^d)$.

Now we can generalize \cite[Corollary~11.11]{NW10}. From Corollary~\ref{cor1} we obtain:

\begin{cor}\label{cor3}
For $a \in (0,1)$ and for $p \in [1,\infty)$ the $L_p$-quadrant discrepancy at $a$ suffers from the curse of dimensionality, i.e., 
\begin{align*}
\lefteqn{N_{L_p^{\boxplus,(a)}}(\varepsilon,d)}\\
& := \min\{N \in \N\ : \ \exists \, \cP \subseteq [0,1)^d \mbox{ s.t. } |\cP|=N \mbox{ and } L_p^{\boxplus,(a)}(\cP)\le \varepsilon L_p^{\boxplus,(a)}(\emptyset)\},
\end{align*}
grows exponentially fast with the dimension $d$. More detailed, $$N_{L_p^{\boxplus,(a)}}(\varepsilon,d) \ge C_{p,1}^d \, (1-\varepsilon) \quad \mbox{for all $\varepsilon \in (0,1)$ and $d \in \N$,}$$ where $C_{p,1}>1$ is like in \eqref{def:Cpr} with $r=1$.
\end{cor}

Also here we point out that in this result even the endpoint $p=1$ is included. For $a=1$ the $L_p$-quadrant discrepancy at 1 is the same as the usual $L_p$-star discrepancy, but, unfortunately, also this case is not covered by Corollary~\ref{cor3}.

\subsection{Weighted integration of functions of smoothness $r$}

For $q \in (1,\infty]$ and $r \in \N$ let 
\begin{align*}
F_1=W_{0,q}^r(\R)=\{f:\R \rightarrow \R \ : \ & f^{(j)}(0)=0\, \mbox{ for $0 \le j <r$},\\
& f^{(r-1)}\, \mbox{ abs. cont. and }\, f^{(r)} \in L_q(\R)\}.
\end{align*}
As norm on $F_1$ we use $$\|f\|_{1,q}:=\|f^{(r)}\|_{L_q}.$$

We study weighted integration in the $d$-fold tensor product space $F_d=W_{0,q}^{(r,r,\ldots,r)}(\R^d)$ equipped with the  crossnorm $\|f\|_{d,q}:=\|f^{(r,r,\ldots,r)}\|_{L_q}$ with respect to a probability density $\psi:\R \to \R_0^+$, i.e., $$I_d(f)=\int_{\R^d} f(\bst) \psi(\bst)\rd \bst \quad \mbox{for $f \in F_{d,q,r}$}$$ with $\psi(\bst):=\psi(t_1)\cdots \psi(t_d)$ for $\bst=(t_1,\ldots,t_d)$. For $q=2$ it is shown in \cite{NW01} and in \cite[Sec.~11.4.1, pp.~169ff]{NW10} that the integration problem suffers from the curse of dimensionality. With our method we are now able to tackle the case $q \in (1,\infty]$. 

Again we observe that the norm $\|\cdot\|_{1,q}$ as well as the crossnorm $\|\cdot\|_{d,q}$ satisfy the $q$-property with the same parameter $q$ like in the definition of the norm. That is, we aim to apply Theorem~\ref{thm2}.

In the following we only explicate the case $q \in (1,\infty)$. The case $q=\infty$ can be treated similarly, incorporating the usual adaptions, like in Section~\ref{susec:3.2}.  

For the sake of simplicity we assume that $\psi$ is symmetric around $0$, i.e., $\psi(x)=\psi(-x)$. Furthermore, we assume that $\psi$ is such that the iterated integral 
\begin{equation}\label{cond:psi}
\int_{0}^{\infty} \int_0^t \int_0^{y_r}\ldots \int_0^{y_2} \left( \int_0^{\infty} (x-y_1)^{r-1} {\bf 1}_{[y_1,\infty)}(x) \psi(x)\rd x \right)^{p-1} \rd y_1 \ldots \rd y_r  \psi(t)\rd t < \infty,
\end{equation}
where $p$ is the H\"older conjugate of $q$. This assumption guarantees that the integrals that appear in the following are well defined.\\

At first we determine the worst-case function. For $x \in \R$ we have (in the same way as above) $f(x)=\int_0^x f^{(r)}(t) \frac{(x-t)^{r-1}}{(r-1)!} \rd t$ (where $\int_0^x f^{(r)}(t) \frac{(x-t)^{r-1}}{(r-1)!} \rd t=-\int_x^0 f^{(r)}(t) \frac{(x-t)^{r-1}}{(r-1)!} \rd t$ whenever $x <0$) and hence
\begin{align*}
\int_{\R} f(x) \psi(x) \rd x = &  - \int_{-\infty}^0 f^{(r)}(t) \int_{-\infty}^t  \frac{(x-t)^{r-1}}{(r-1)!} \psi(x) \rd x \rd t\\
& +  \int_0^{\infty} f^{(r)}(t) \int_t^{\infty}  \frac{(x-t)^{r-1}}{(r-1)!} \psi(x) \rd x \rd t.
\end{align*}
Let $$\Psi_r(t):=\left\{ 
\begin{array}{ll}
-\int_{-\infty}^0  \frac{(x-t)^{r-1}}{(r-1)!} {\bf 1}_{(-\infty,t]}(x)  \psi(x) \rd x\quad \mbox{for $t < 0$,}\\[1em]
\int_{0}^{\infty} \frac{(x-t)^{r-1}}{(r-1)!} {\bf 1}_{[t,\infty)}(x)  \psi(x) \rd x\quad \mbox{for $t \ge 0$,}
\end{array}
\right.
$$
(note that $\Psi_r(-t)=(-1)^r \Psi_r(t)$.) Then 
\begin{align*}
\int_{\R} f(x) \psi(x) \rd x =  \int_{\R} f^{(r)}(t) \Psi_r(t) \rd t.
\end{align*}
Applying H\"older's inequality, where $p$ is the H\"older conjugate of $q$, we obtain 
$$\left|\int_{\R} f(x)\psi(x) \rd x \right| \le \|f^{(r)}\|_{L_q} \|\Psi_r\|_{L_p}=\|f\|_{1,q} \|\Psi_r\|_{L_p}$$ with equality if $|f^{(r)}(t)|=c |\Psi_r|^{p-1}$ for some $c >0$. This holds for $f$ of the form
\begin{equation}\label{defh1:r}
f(t)= \int_0^{|t|} \int_0^{y_r} \int_0^{y_{r-1}}\ldots \int_0^{y_2} \left(\Psi_r(y_1)\right)^{p-1} \rd y_1 \ldots \rd y_{r-1} \rd y_r.
\end{equation}
This means that $h_1:=f$ is the worst-case function. Note that $h_1(t)=h_1(-t)$. 

We use the decomposition of $h_1$ with $$h_{1,(0)}(x)={\bf 1}_{(-\infty,0]}(x)h_1(x) \quad \mbox{and}\quad h_{1,(1)}(x)={\bf 1}_{[0,\infty)}(x) h_1(x).$$ In particular, we have $D_1=\R$, $D_{(0)}=(-\infty,0]$ and $D_{(1)}=[0,\infty)$. 
We assumed that $\psi$ is symmetric around 0, in particular, $\psi$ does not vanish over $(-\infty,0]$ and $[0,\infty)$. Hence 
\begin{equation*}
\int_{-\infty}^0 h_1(t)\psi(t)\rd t = \int_0^{\infty} h_1(t)\psi(t)\rd t= \frac{1}{2}  \int_{-\infty}^{\infty} h_1(t)\psi(t)\rd t >0.
\end{equation*}
Note that these integrals exist because of assumption \eqref{cond:psi}. So we meet all required properties~{\bf D1}-{\bf D3} and $$\frac{1}{\alpha}=\frac{I_1(h_{1,(0)}) + I_1(h_{1,(1)})}{\max(I_1(h_{1,(0)}),I_1(h_{1,(1)}))}=2.$$ Thus Theorem~\ref{thm2} implies the following result.

\begin{cor}
Let $q \in (1,\infty]$, $r \in \N$ and $\psi$ be symmetric around $0$ and assume that $\psi$ satisfies \eqref{cond:psi}. Then $\psi$-weighted integration in $W_{0,q}^{(r,r,\ldots,r)}(\R^d)$ suffers from the curse of dimensionality, in more detail, we have $$N(\varepsilon,d) \ge 2^d \, (1-\varepsilon) \quad \mbox{for all $\varepsilon \in (0,1)$ and $d \in \N$.}$$ 
\end{cor}

That is, we generalized the example in \cite[Sec. 11.4.1, pp.~169f]{NW10} from $q=2$ to general $q \in (1,\infty]$.

\section{Positive quadrature rules}\label{sec:meth3}

The big issue in the application of the methods presented so far is to check whether the fooling function $g_d$ satisfies $\|g_d\|_d \le \|h_d\|_d$. This is guaranteed by means of the rather general conditions {\bf D4} and {\bf DP4}, respectively, which however are in most cases very difficult to verify. This problem can be overcome for positive linear rules $A_{N,d}^+$ like in \eqref{def:linAlg}, but only with non-negative weights $a_1,\ldots,a_N \ge 0$. We present a method that is based on spline approximations of the worst-case function. In order to state the results we slightly adapt the definitions of the minimal error and the information complexity. Define
$$e^+(N,d):=\min_{A_{d,N}^+} e(F_d, A_{d,N}^+),$$ where here the minimum is extended over all linear algorithms of the form \eqref{def:linAlg} based on $N$ function evaluations along points $\bsx_1,\bsx_2,\ldots,\bsx_N$ from $D_d$ and non-negative $a_1,\ldots,a_N \ge 0$.

For $\varepsilon \in (0,1)$ and $d \in \mathbb{N}$ define $$N^+(\varepsilon,d)=\min\{N \in \mathbb{N} \ : \ e^+(N,d) \le \varepsilon \, e(0,d)\}.$$ We stress that $N^+(\varepsilon,d)$ is a kind of restricted complexity since we only allow positive quadrature formulas.

Now we can state our main theorem for this section:

\begin{thm}\label{thm5}
Let $(F_1,\|\cdot\|_1)$ be a normed space of univariate functions over $D_1 \subseteq \R$ and let for $d \in \N$, $(F_d,\|\cdot\|_d)$ be the $d$-fold tensor product space equipped with a crossnorm $\|\cdot\|_d$. Assume in $F_1$ exists a worst-case function $h_1$.

Assume that for every $y \in D_1$ there exists a function $s_y \in F_1$ such that $s_y \ge 0$ and $s_y(y)=h_1(y)$ and 
\begin{equation}\label{est:alphabeta}
\alpha:= \max_{y \in D_1} \|s_y\|_1 < \|h_1\|_1\quad \mbox{and}\quad \beta:= \max_{y \in D_1} I_1(s_y) < I_1(h_1).
\end{equation} 

Put
\begin{equation}\label{def:C1new}
\widetilde{C}:=\min\left(\frac{\|h_1\|_1}{\alpha},\frac{I_1(h_1)}{\beta}\right).
\end{equation}
Then $\widetilde{C}>1$ and $$N^+(\varepsilon,d) \ge \widetilde{C}^d\, (1-2 \varepsilon) \quad \mbox{for all $\varepsilon \in \left(0,\frac{1}{2}\right)$ and $d \in \N$.}$$ In particular, the integration problem in $F_d$ suffers from the curse of dimensionality for positive quadrature formulas.
\end{thm}

\begin{proof}
Consider a positive linear algorithm $A_{d,N}^+$ based in nodes $\bsx_1,\ldots,\bsx_N$ in $D_d$ and with weights $a_1,\ldots,a_N\ge 0$.  Let $x_{i,j}$ be the $j$-th coordinate of the point $\bsx_i$, $i \in \{1,\ldots,N\}$ and $j \in [d]$. For $i \in \{1,\ldots,N\}$ we define functions 
$$
P_i(\bsx) = s_{x_{i,1}}(x_1) s_{x_{i,2}}(x_2) \cdots s_{x_{i,d}}(x_d),\quad \bsx=(x_1,\ldots,x_d)\in D_d.  
$$

In order to estimate the error of $A_{d,N}^+$ we consider the two functions $h_d$ and $f^*= \sum_{i=1}^N P_i$. Since $A_{d,N}^+$ is a positive rule we have $$A_{d,N}^+(f^*) \ge A_{d,N}^+(h_d).$$ Then we have the error estimate 
\begin{equation}\label{errest1linneu}
e(F_d,A_{d,N}^+)  \ge \frac{ \left(I_d(h_d) - I_d(f^*)\right)_+}{2 \max ( \Vert h_d \Vert_d, \Vert f^* \Vert_d)},
\end{equation}
which is trivially true if $I_d(h_d) \le  I_d(f^*)$ and which is easily shown if $I_d(h_d) >  I_d(f^*)$, because then 
\begin{align*}
\left(I_d(h_d) - I_d(f^*)\right)_+  \le & I_d(h_d) - A_{d,N}^+(h_d)+A_{d,N}^+(f^*)-I_d(f^*)\\
 \le & \|h_d\|_d \, e(F_d,A_{d,N}^+)+\|f^*\|_d \, e(F_d,A_{d,N}^+)\\
 \le & 2 \max(\|h_d\|_d,\|f^*\|_d) \, e(F_d,A_{d,N}^+).
\end{align*}

According to \eqref{est:alphabeta} we have $\Vert f^* \Vert_d \le N \alpha^d$ and $I_d(f^*) \le N  \beta^d$. Inserting into \eqref{errest1linneu} yields 
\begin{equation}\label{lberr17}
e^{+}(N,d) \ge e(F_d,A_{d,N}^+) \ge \frac{(I_d(h_d) - N \beta^d)_+} {2 \max(\|h_d\|_d,N \alpha^d)} .
\end{equation}

Now let $\varepsilon \in (0,1/2)$ and assume that $e^+(N,d) \le \varepsilon \, e(0,d)$. This implies 
\begin{equation}\label{lberr17a}
2 \, \varepsilon \, e(0,d) \, \max(\|h_d\|_d,N \alpha^d) \ge  (I_d(h_d) - N \beta^d)_+.
\end{equation}

Consider $\widetilde{C}$ from \eqref{def:C1new}. According to \eqref{est:alphabeta} we have $\widetilde{C}>1$.

If $N \le \widetilde{C}^d$, then we obtain $\max(\|h_d\|_d,N \alpha^d)=\|h_d\|_d$ and $I_d(h_d) \ge N \beta^d$ and hence \eqref{lberr17a} implies
\begin{align*}
I_d(h_d) - N \beta^d = &  (I_d(h_d) - N \beta^d)_+ \le 2 \, \varepsilon \, e(0,d)\, \|h_d\|_d = 2 \, \varepsilon \, I_d(h_d). 
\end{align*}
Hence $$N \ge \frac{I_d(h_d)}{\beta^d} \, (1-2\varepsilon)= \left(\frac{I_1(h_1)}{\beta}\right)^d (1-2 \varepsilon) \ge \widetilde{C}^d \, (1-2 \varepsilon).$$

If $N > \widetilde{C}^d$, then we trivially have $N \ge \widetilde{C}^d \, (1-2 \varepsilon)$. 
This yields $$N^+(\varepsilon,d)\ge \widetilde{C}^d \, (1-2 \varepsilon),$$ and we are done.
\end{proof} 

\begin{remark} \rm
Let us compare Theorem 4 with recent lower bounds 
from \cite{HKNV21,HKNV22,KV23} 
that are based on a new result of 
Vyb{\'\i}ral~\cite{Vy20}
who improved the 
Schur product theorem. 

Both lower bounds do not need a decomposition of $h_1$ 
and therefore can be applied to classes 
of analytic functions or even polynomials. 
Our result can be applied to more functionals
and more spaces, while the lower bounds
in the papers cited above  
only work for particular functionals and 
Hilbert spaces of functions. 
Our approach has the disadvantage, however, 
that it only works for positive quadrature formulas, 
such as quasi Monte Carlo methods. 
\end{remark}

Sometimes one has a decomposition of the worst-case function in the sense of Section~\ref{sec:meth1.1} and \ref{sec:meth2}. I.e., the worst-case function $h_1$ exists and can be decomposed in the form 
\begin{equation}\label{dech:partpos}
h_1(x)=h_{1,1}(x)+h_{1,2,(0)}(x)+h_{1,2,(1)}(x)
\end{equation}
with the following properties:
\begin{description}
\item[{\bf DP}$\boldsymbol{1}^+$] $h_{1,1}, h_{1,2,(0)}, h_{1,2,(1)} \in F_1$ are non-negative;  
\item[{\bf DP}$\boldsymbol{2}^+$] there exists $a \in \R$ such that the sets $$D_{(0)}=\{x \in D_1 \ : \ x \le a\} \quad \mbox{and}\quad D_{(1)}=\{x \in D_1 \ : \ x \ge a\}$$ are non-empty and the support of $h_{1,2,(0)}$ is contained in $D_{(0)}$ and the support of $h_{1,2,(1)}$ is contained in $D_{(1)}$; 
\item[{\bf DP}$\boldsymbol{3}^+$] the integrals $I_1(h_{1,2,(0)})$ and $I_1(h_{1,2,(1)})$ are positive;
\item[{\bf DP}$\boldsymbol{4}^+$] for $\alpha:=\max\left(\Vert h_{1,1} + h_{1,2,(0)} \Vert_1, \Vert h_{1,1} + h_{1,2,(1)} \Vert_1\right)$ we have $\alpha < \|h_1\|_1$.
\end{description}

Note that $h_{1,1}$ might also be 0 on $D_1$. Then we have the following corollary to Theorem~\ref{thm5}.

\begin{cor}\label{thm4}
Let $(F_1,\|\cdot\|_1)$ be a normed space of univariate functions over $D_1 \subseteq \R$ and let for $d \in \N$, $(F_d,\|\cdot\|_d)$ be the $d$-fold tensor product space equipped with a crossnorm $\|\cdot\|_d$. Assume in $F_1$ exists a worst-case function $h_1$ which can be decomposed according to \eqref{dech:partpos} satisfying properties~${\bf DP1^+}$-${\bf DP4^+}$. Put 
\begin{equation}\label{def:C}
C:=\min\left(\frac{\|h_1\|_1}{\alpha},\frac{I_1(h_1)}{\beta}\right),
\end{equation}
where $\beta := I(h_{1,1})+\max(I_1(h_{1,2,(0)}) , I_1(h_{1,2,(1)}))$.

Then $C>1$ and $$N^+(\varepsilon,d) \ge C^d\, (1-2 \varepsilon) \quad \mbox{for all $\varepsilon \in \left(0,\frac{1}{2}\right)$ and $d \in \N$.}$$ In particular, the integration problem in $F_d$ suffers from the curse of dimensionality for positive quadrature formulas.
\end{cor}

\begin{proof}
Choose in Theorem~\ref{thm5} $$s_y(x)=h_{1,1}(x)+\left\{
\begin{array}{ll}
h_{1,2,(0)}(x) & \mbox{if $y \in D_{(0)}$,}\\
h_{1,2,(1)}(x) & \mbox{if $y \in D_{(1)}$.}
\end{array}
\right.$$
Then $s_y(y)=h_1(y)$ for every $y \in D_1$ and \eqref{def:C1new} is satisfied because of ${\bf DP3^+}$ and ${\bf DP4^+}$. Hence, the result follows from Theorem~\ref{thm5}.
\end{proof}

\section{Applications: Part 2}\label{sec:exa2}

An important application of the method for positive quadrature rules from Section~\ref{sec:meth3} can be found in \cite{NP24} where we proved that the classical $L_p$-star discrepancy for every $p \in (1,\infty)$ suffers from the curse of dimensionality. Here we provide some further applications.

\subsection{Integration of polynomials of degree at most~2}\label{exa:poly2}

Consider $F_1:=P_2$ as the space of all real polynomials of degree at most 2 as functions over $[0,1]$, equipped with the $q$-norm $$\|f\|_q:=\left(\|f\|_{L_q}^q+\|f'\|_{L_q}^q+\|f''\|_{L_q}^q\right)^{1/q}$$ for some $q \in [1,\infty)$. Consider integration $I_1(f)=\int_0^1f(x)\rd x$ for $f \in F_1$. This problem is studied in \cite[Section~10.5.4]{NW10}, but only for the Hilbert-space case $q=2$.

The initial error for this problem is $e(0,1)=1$ and the worst-case function is $h_1(x)=1$ for $x \in [0,1]$. Obviously $\|h_1\|_q=1$ and $\int_0^1 h_1(x)\rd x =1$. 

For $y \in [0,1]$ choose $s_y(x):=1-c(y-x)^2$ with some $c \in (0,1/2)$ that will be specified later. Then $s_y \in F_1$ and $s_y \ge 0$ for every $y \in [0,1]$. Furthermore, $$I_1(s_y)=1-c\left(y^2-y+\frac{1}{3}\right)$$ and hence $\beta=1-\frac{c}{12}<1$.

We estimate $\|s_y\|_q$. Since $s_y(x) \in [0,1]$ and $q \ge 1$ we have $$\|s_y\|_{L_q}^q =\int_0^1 |s_y(x)|^q \rd x \le I_1(s_y) \le \beta = 1-\frac{c}{12}.$$ Furthermore, since $s_y'(x)=2 c (y-x)$ and $s_y''(x)=-2c$ we have $$\|s_y'\|_{L_q}^q=\int_0^1|2 c(y-x)|^q \rd x=(2 c)^q \frac{y^{q+1}+(1-y)^{q+1}}{q+1} \le \frac{(2 c)^q}{q+1}$$ and $$\|s_y''\|_{L_q}^q=\int_0^1|-2 c|^q \rd x = (2c)^q.$$ Hence $$\|s_y\|_q^q \le  1-\frac{c}{12} + (2 c)^q \frac{q+2}{q+1}.$$ For $q \in (1,\infty)$  the latter is strictly less than 1, if 
\begin{equation}\label{def:uq}
c< \left(\frac{1}{12 \cdot 2^q} \frac{q+1}{q+2}\right)^{1/(q-1)}=:u_q.
\end{equation}
Thus for the quantity in \eqref{def:C1new} we have
\begin{eqnarray*}
\widetilde{C}_q & = & \max_{c \in (0,u_q)}\min\left(\frac{1}{\left(1-\frac{c}{12} + (2 c)^q \frac{q+2}{q+1}\right)^{1/q}},\frac{1}{1-\frac{c}{12}}\right)\\
& = & \max_{c \in (0,u_q)}\left(1-\frac{c}{12} + (2 c)^q \frac{q+2}{q+1}\right)^{-1/q}.
\end{eqnarray*}
It is easily checked that the maximum is attained for $$c=\left(\frac{1}{q}\frac{1}{12 \cdot 2^q} \frac{q+1}{q+2}\right)^{1/(q-1)}=\frac{u_q}{q^{1/(q-1)}}.$$  Hence, with $u_q$ taken from \eqref{def:uq}, 
\begin{equation}\label{def:Cqpoly}
\widetilde{C}_q=\left(1-\frac{u_q}{12 q^{1/(q-1)}} + \frac{2^q u_q^q}{q^{q/(q-1)}} \frac{q+2}{q+1}\right)^{-1/q}.
\end{equation}

Now Theorem~\ref{thm5} implies:

\begin{cor}\label{corP2}
For $q \in (1,\infty)$ the integration problem in the tensor-product space based on $F_1=P_2$ suffers from the curse of dimensionality for positive quadrature formulas. More detailed, we have $N^+(\varepsilon,d) \ge \widetilde{C}_q^d \, (1-2\varepsilon)$ for all $\varepsilon \in (0,1/2)$ and $d \in \N$, where $\widetilde{C}_q$ is from \eqref{def:uq} and hence $\widetilde{C}_q>1$.
\end{cor}

Exemplarily, we have the following (rounded) figures for $\widetilde{C}_q$:
$$
\begin{array}{r||c|c|c|c|c|c|c}
q & 2 & 3 & 4 & 5 & 10 & 100 & 1000\\
\hline
\widetilde{C}_q & 1.00016 & 1.00098 & 1.00161 & 1.00195 & 1.00204 & 1.00039 & 1.00004
\end{array}
$$
and $\lim_{q \rightarrow \infty} \widetilde{C}_q = 1$.

\subsection{Integration of $C^{\infty}$-functions}

Let $F_1$ be the space of $C^\infty ([0,1])$ functions such that the norm 
 $\|f\|_q:=(\sum_{\alpha =0}^{\infty}
  \|f^{(\alpha)}\|^q_{L_q})^{1/q}$ is finite, where $q \in [1,\infty)$. 
  The worst-case error of integration in $F_1$ is $e(0,1)=1$ and the worst-case function is again $h_1(x)=1$ for $x \in [0,1]$. It follows from Corollary~\ref{corP2} that for $q \in (1,\infty)$ integration in the corresponding tensor-product space suffers from the curse of dimensionality for positive quadrature formulas.
 
We do not know whether this is also true for $q=\infty$, i.e.,  for the norm $\|f\|_\infty := \max_{\alpha} \|f^{(\alpha)}\|_\infty $.

\subsection{Uniform integration without anchor condition}

As yet another application we revisit the problem of uniform integration from Section~\ref{susec:3.2}, but now we remove the anchor condition $f(a)=0$. For the sake of simplicity we restrict ourselves to the case $r=1$ and $q \in (1,\infty)$. Let 
\begin{align*}
F_1=W_q^1([0,1])=\{f:[0,1] \rightarrow \R \ : \ f\, \mbox{ abs. cont. and }\, f' \in L_q([0,1])\}
\end{align*}
equipped with the norm $$\|f\|_{1,q}=\left(|f(a)|^q+\int_0^1 |f'(t)|^q \rd t \right)^{1/q}$$ for some fixed $a \in (0,1)$. 
We study integration in the $d$-fold tensor product space $F_d=W_q^{(1,1,\ldots,1)}([0,1]^d)$ equipped with the crossnorm $$\|f\|_{d,q}=\left(\sum_{\uu \subseteq [d]} \|f_{\uu}(\cdot_\uu,a)\|_{L_q}^q\right)^{1/q},$$ where $f_{\uu}(\cdot_\uu,a)$ denotes the partial derivative of $f$ of order 1 with respect to the coordinates that belong to $\uu$ and all other components set to be $a$.  For $q=2$ it is shown in \cite[Sec.~11.5.2, pp.~188f]{NW10} that the integration problem suffers from the curse of dimensionality. We tackle the more general case and aim to apply Corollary~\ref{thm4}.

Let again $p$ be the H\"older conjugate of $q$. Using methods similar to Section~\ref{susec:3.2} it is easy to prove that the worst-case function for the integration problem in $F_1$ is 
\begin{equation}\label{h1:F1qa}
h_1(x)=1+\boldsymbol{1}_{[0,a]}(x) \frac{a^p - x^p}{p}+\boldsymbol{1}_{[a,1]}(x) \frac{(1-a)^p-(1-x)^p}{p}.
\end{equation}
We have $$\int_0^1 h_1(x)\rd x=1+\frac{a^{p+1}+(1-a)^{p+1}}{p+1}$$ and $$\|h_1\|_{1,q}=\left(1+\frac{a^{p+1}+(1-a)^{p+1}}{p+1}\right)^{1/q}.$$ The initial error is $$e(0,1)=\left(1+\frac{a^{p+1}+(1-a)^{p+1}}{p+1}\right)^{1/p}.$$

We immediately see the suitable decomposition of $h_1$, namely 
\begin{eqnarray*}
h_{1,1}(x) & = & 1\\
h_{1,2,(0)}(x) & = & \boldsymbol{1}_{[0,a]}(x) \frac{a^p - x^p}{p}\\
h_{1,2,(1)}(x) & = & \boldsymbol{1}_{[a,1]}(x) \frac{(1-a)^p-(1-x)^p}{p}.
\end{eqnarray*}
Then ${\bf DP1^+}$ and ${\bf DP2^+}$ are satisfied, furthermore $$\int_0^1 h_{1,2,(0)}(x) \rd x = \frac{a^{p+1}}{p+1}>0\quad \mbox{ and }\quad \int_0^1 h_{1,2,(1)}(x) \rd x = \frac{(1-a)^{p+1}}{p+1}>0$$ 
such that ${\bf DP3^+}$ is satisfied. Finally, since $a \in (0,1)$, 
$$\|h_{1,1}+h_{1,2,(0)}\|_{1,q}^q = 1+\frac{a^{p+1}}{p+1} < \|h_1\|_{1,q}^q$$ and $$\|h_{1,1}+h_{1,2,(1)}\|_{1,q}^q = 1+\frac{(1-a)^{p+1}}{p+1} < \|h_1\|_{1,q}^q,$$ such that also ${\bf DP4^+}$ is satisfied. Thus Corollary~\ref{thm4} implies:

\begin{cor}\label{cor5}
For $q\in (1,\infty)$ and $a\in (0,1)$ integration in the tensor product space $W_q^{(1,1,\ldots,1)}([0,1]^d)$ suffers from the curse of dimensionality for positive quadrature formulas. In more detail, we have $$N^+(\varepsilon,d) \ge C_p^d \,(1-2 \varepsilon)\quad \mbox{for all $\varepsilon \in \left(0,\frac{1}{2}\right)$ and $d \in \N$,}$$ where $p$ is the H\"older conjugate of $q$ and where 
\begin{equation}\label{def:Cap}
C_p=\left(\frac{p+1+a^{p+1}+(1-a)^{p+1}}{p+1+\max(a^{p+1},(1-a)^{p+1})}\right)^{1/q}>1.
\end{equation}
\end{cor}

For example for $a=1/2$ we have
$$
\begin{array}{r||c|c|c|c|c}
p & 2 & 3 & 4 & 5 & 10\\
\hline
C_p & 1.0198 & 1.01023 & 1.00465 & 1.00208 & 1.00004
\end{array}
$$

\subsection{Generalized $L_p$-discrepancy anchored at $a$}
For a point set $\{\bsy_1,\ldots,\bsy_N\}$ in $[0,1)^d$ the generalized $L_p$-discrepancy anchored at $a$ is defined as $$L_p^{\pitchfork, (a), {\rm gen}}(\{\bsy_1,\ldots,\bsy_N\}):=\left(\sum_{\uu \subseteq [d]} \left(L_p^{\pitchfork, (a)}(\{\bsy_{1,\uu},\ldots,\bsy_{N,\uu}\}) \right)^p \right)^{1/p},$$ where $L_p^{\pitchfork, (a)}$ denotes the $L_p$-discrepancy anchored in $a$ from Section~\ref{subsec3.3} and where, for $\uu \subseteq [d]$, $\bsy_{k,\uu}$ denotes the projection of $\bsy_k$ to the coordinates $j$ in $\uu$. 

Consider integration in the tensor product space $W_q^{(1,1,\ldots,1)}([0,1]^d)$. It can be shown that the worst-case error of a quasi-Monte Carlo rule \eqref{QMC} for $f \in W_q^{(1,1,\ldots,1)}([0,1]^d)$, which is based on a point set $\{\bsx_1,\ldots,\bsx_N\}$, is exactly the generalized $L_p$-discrepancy anchored at $a$ of the point set $\{\bsy_1,\ldots,\bsy_N\}$, where $\bsy_k:=\bsa-\bsx_k \pmod{1}$ and where $\bsa=(a,\ldots,a)$, i.e., $$e(W_q^{(1,1,\ldots,1)}([0,1]^d),A_{d,N}^{{\rm QMC}})=L_p^{\pitchfork, (a), {\rm gen}}(\{\bsy_1,\ldots,\bsy_N\}),$$ where $1/p+1/q=1$. See \cite[Section~9.5.3]{NW10}. The initial generalized $L_p$-discrepancy anchored at $a$ is $$L_p^{\pitchfork, (a), {\rm gen}}(\emptyset)=\left(1+\frac{a^{p+1}+(1-a)^{p+1}}{p+1}\right)^{d/p}=e(0,d),$$ where the latter is the initial error of integration in $W_q^{(1,1,\ldots,1)}([0,1]^d)$.

From Corollary~\ref{cor5} we therefore obtain:

\begin{cor}
For $a \in (0,1)$ and for $p \in (1,\infty)$ the generalized $L_p$-discrepancy anchored at $a$ suffers from the curse of dimensionality, i.e., 
\begin{align*}
\lefteqn{N_{L_p^{\pitchfork, (a),{\rm gen}}}(\varepsilon,d)}\\
&:= \min\{N \in \N\ : \ \exists \, \cP \subseteq [0,1)^d \mbox{ s.t. } |\cP|=N \mbox{ and } L_p^{\pitchfork, (a),{\rm gen}}(\cP)\le \varepsilon L_p^{\pitchfork, (a),{\rm gen}}(\emptyset)\},
\end{align*}
grows exponentially fast with the dimension $d$. More detailed, $$N_{L_p^{\pitchfork, (a),{\rm gen}}}(\varepsilon,d) \ge C_p^d\, (1-2 \varepsilon) \quad \mbox{for all $\varepsilon \in \left(0,\frac{1}{2}\right)$ and $d \in \N$,}$$ where $C_p>1$ is like in \eqref{def:Cap}.
\end{cor}

\subsection{Generalized $L_p$-quadrant discrepancy at $a$}
Likewise, also a generalized $L_p$-quadrant discrepancy at $a$ is studied. For a point set $\{\bsy_1,\ldots,\bsy_N\}$ in $[0,1)^d$ the generalized $L_p$-quadrant discrepancy at $a$ is defined as $$L_p^{\boxplus, (a), {\rm gen}}(\{\bsy_1,\ldots,\bsy_N\}):=\left(\sum_{\uu \subseteq [d]} \left(L_p^{\boxplus, (a)}(\{\bsy_{1,\uu},\ldots,\bsy_{N,\uu}\}) \right)^p \right)^{1/p},$$ where $L_p^{\boxplus, (a)}$ denotes the $L_p$-quadrant discrepancy at $a$ from Section~\ref{subsec3.4} and where, for $\uu \subseteq [d]$, $\bsy_{k,\uu}$ denotes the projection of $\bsy_k$ to the coordinates $j$ in $\uu$. 

Again, it can be shown that the worst-case error of a quasi-Monte Carlo rule \eqref{QMC} for $f \in W_q^{(1,1,\ldots,1)}([0,1]^d)$, which is based on a point set $\{\bsx_1,\ldots,\bsx_N\}$, is exactly the generalized $L_p$-quadrant discrepancy at $a$ of the same point set, i.e., $$e(W_q^{(1,1,\ldots,1)}([0,1]^d),A_{d,N}^{{\rm QMC}})=L_p^{\boxplus,(a),{\rm gen}}(\{\bsx_1,\ldots,\bsx_N\}),$$ where $1/p+1/q=1$. See \cite[Section~9.5.4]{NW10} and \cite[Sec.~11.5.3]{NW10}. The initial generalized $L_p$-quadrant discrepancy at $a$ is $$L_p^{\boxplus, (a),{\rm gen}}(\emptyset)=\left(1+\frac{a^{p+1}+(1-a)^{p+1}}{p+1}\right)^{d/p}=e(0,d),$$ where the latter is the initial error of integration in $W_q^{(1,1,\ldots,1)}([0,1]^d)$.

From Corollary~\ref{cor5} we therefore obtain:

\begin{cor}
For $a \in (0,1)$ and for $p \in (1,\infty)$ the generalized $L_p$-quadrant discrepancy at $a$ suffers from the curse of dimensionality, i.e., 
\begin{align*}
\lefteqn{N_{L_p^{\boxplus,(a),{\rm gen}}}(\varepsilon,d)}\\
&:= \min\{N \in \N\ : \ \exists \, \cP \subseteq [0,1)^d \mbox{ s.t. } |\cP|=N \mbox{ and } L_p^{\boxplus, (a),{\rm gen}}(\cP)\le \varepsilon L_p^{\boxplus,(a),{\rm gen}}(\emptyset)\},
\end{align*}
grows exponentially fast with the dimension $d$. More detailed, $$N_{L_p^{\boxplus,(a),{\rm gen}}}(\varepsilon,d) \ge C_p^d\, (1-2 \varepsilon) \quad \mbox{for all $\varepsilon \in \left(0,\frac{1}{2}\right)$ and $d \in \N$,}$$ where $C_p>1$ is like in \eqref{def:Cap}.
\end{cor}

\vspace{0.5cm}
\noindent{\bf Author's Address:}\\

\noindent Erich Novak, Mathematisches Institut, FSU Jena, Ernst-Abbe-Platz 2, 07740 Jena, Germany. Email: erich.novak@uni-jena.de\\

\noindent Friedrich Pillichshammer, Institut f\"{u}r Finanzmathematik und Angewandte Zahlentheorie, JKU Linz, Altenbergerstra{\ss}e 69, A-4040 Linz, Austria. Email: friedrich.pillichshammer@jku.at

\end{document}